\documentclass[11pt,reqno,twoside]{article}
\usepackage[utf8]{inputenc}
\usepackage{amsmath}
\usepackage{amsfonts}
\usepackage{amssymb}
\usepackage{amsthm}	
\usepackage{cases}
\usepackage{subeqnarray}
\usepackage{cite}
\usepackage[colorlinks]{hyperref}
\usepackage{graphics}
\usepackage{graphicx}
\usepackage{fancyref}
\usepackage{epstopdf}
\usepackage{float}
\usepackage{listings}
\usepackage{lipsum}
\usepackage{titlesec}
\usepackage[titletoc,toc,title]{appendix}
\usepackage{enumerate}

\renewcommand{\oddsidemargin}{5 mm}

\makeatletter
\allowdisplaybreaks
\numberwithin{equation}{section}
\newtheorem{theorem}{Theorem}[section]
\newtheorem{corollary}[theorem]{Corollary}
\newtheorem{lemma}[theorem]{Lemma}

\newtheorem{definition}[theorem]{Definition}
\newtheorem{remark}{Remark}
\allowdisplaybreaks \numberwithin{remark}{section}
\begin{document}
\author{ Wan-Tong Li$^1$, Jia-Bing Wang$^{1,2,}$, ~Xiao-Qiang Zhao$^2$ \\
$^1$School of Mathematics and Statistics, Lanzhou University,\\
Lanzhou, Gansu 730000, People's Republic of China\\
$^2$Department of Mathematics and Statistics, Memorial University \\of Newfoundland,
St. John's, NL A1C 5S7, Canada}

\title{\textbf{Spatial Dynamics of a Nonlocal Dispersal Population Model
in a Shifting Environment } }

\date{}

\maketitle

\begin{abstract}
This paper is concerned with spatial spreading dynamics of a
nonlocal dispersal population model in a shifting environment
where the favorable region is shrinking.
It is shown that the species will become extinct in the habitat
once the speed of the shifting habitat edge $c>c^*(\infty)$,
however if $c<c^*(\infty)$,
the species will persist and spread along the shifting habitat
at an asymptotic spreading speed $c^*(\infty)$, where $c^*(\infty)$ is determined
by the nonlocal dispersal kernel, diffusion rate
and the maximum linearized growth rate.
Moreover, we demonstrate that for any given speed of the shifting habitat edge,
this model admits a nondecreasing traveling wave
with the wave speed at which the habitat is shifting,
which indicates that the extinction wave phenomenon does happen
in such a shifting environment.


\end{abstract}


\section{Introduction}
In this paper, we are interested in the following nonlocal dispersal population model
in a shifting environment:
\begin{equation}\label{1}
\frac{\partial u(t,x)}{\partial t}=d\left(J\ast u-u\right)(t,x)+u(t,x)(r(x-ct)-u(t,x)),
~t>0, x\in\mathbb{R},
\end{equation}
where $u(t,x)$ stands for the population density of the species under consideration at time $t$
and location $x$. Here the convolution kernel \emph{$J\in C(\mathbb{R},\mathbb{R}^+)$ is even
and compactly supported with unit integral},  and $J(x-y)$ denotes the probability distribution of
the population jumping from location $y$ to location $x$.
Then $\int_{\mathbb{R}}J(x-y)u(t,y){\rm d}y$ is the rate
at which individuals are arriving to location $x$ from all other places,
while $\int_{\mathbb{R}}J(y-x)u(t,x)dy =u(t,x)$ is the rate
at which they are leaving location $x$ to all other sites.
It follows that
$$\mathcal{A}u(t,x):=(J\ast u-u)(t,x)=\int_{\mathbb{R}}J(x-y)u(t,y){\rm d}y-u(t,x)$$
can be viewed as a nonlocal dispersal operator modeling the free and large-range migration
of the species (see \cite{IgnatRossi2007,Murray2003}) and $d>0$ is the dispersal rate.
The reaction term  describes the logistic type growth of the species
which depends on the density $u$ and on the shifting habitat with a fixed speed $c>0$.
Throughout this paper, we always assume that the resource function
\textit{$r(\xi)$ is a continuous and nondecreasing function with
$r(\pm\infty)$ finite and $r(-\infty)<0<r(\infty)$.}
Thus, the shifting environment may be divided into the favourable region
$\{x\in\mathbb{R}: r(x-ct)>0\}$ and the unfavourable region
$\{x\in\mathbb{R}: r(x-ct)\leq0\}$, both shifting with speed $c>0$.
Specifically, we see that when time increases,
the unfavourable region is expanding and the favourable region is shrinking.
This kind of problem comes from considering the threats associated with global climate change
and the worsening of the environment resulting from industrialization
which lead to the shifting or translating of the habitat ranges,
and recently has attracted much attention,
see, e.g., \cite{BerestyckiBMB,GNLD2010,HuZou2017,LeiDu2017,lewis2004,
LiBTsiam,LiBTjde,LiBMB,ZhouKot2011,Parr2012,Scheiter2009}.
Model \eqref{1} may also be derived from some specific epidemiological models
by the arguments similar to those in Fang, Lou and Wu \cite{FangLouWuSIAM},
where the authors deduced a classical reaction-diffusion Fisher-KPP equation
in a wavelike environment from the consideration of pathogen spread.

It is well known that nonlocal dispersal problem \eqref{1}
with space-time homogeneous growth rate $r>0$,
i.e., $u_t=d(J\ast u-u)+u(r-u)$,
has been fully investigated for the spatial spreading dynamics.
Here we refer to \cite{Coville2007, Schumacher1980,Yagisita2009,Carr-Chmaj}
for the existence and uniqueness of monotone traveling wave solutions,
and \cite{Schumacher1980',Lutscher2005} for the spreading speed,
and \cite{LiSunWang2010} for the construction of new types of entire solutions.
Roughly speaking, the slowest speed
$$c^*=\min_{\lambda>0}\frac{d\left(\int_{\mathbb{R}}J(y)
e^{\lambda y}{\rm d}y-1\right)+r}{\lambda}$$
for a class of traveling fronts connecting $r$ and $0$
is of some important spreading properties.
More precisely, let $u(t,x;u_0)$ be the nonnegative solution
with compactly supported initial data $u_0$, then
$\lim_{t\rightarrow\infty, |x|\geq ct}u(t,x;u_0)=0$ for any $c>c^*$
and $\lim_{t\rightarrow\infty, |x|\leq ct}u(t,x;u_0)=r$ for any $0<c<c^*$.
Ecologically, the spreading speed can be understood as the asymptotic rate at
which a species, initially introduced in a bounded domain,
expands its spatial range as time evolves, while
a traveling wave describes the propagation of a species as a wave with a fixed shape and
a fixed speed. These two fundamental issues along with some new types of entire solutions
have been widely used for the description of species invasion and disease transmission.
Regarding the nonlocal dispersal equation with
time and/or space periodic dependence,
we refer the readers to \cite{JDE2010,ProAMS2012, DCDS2015}
for spreading speeds, and \cite{CPNA2012, Pocare2013, DCDS2015,Bates1999}
for traveling wave solutions,
and \cite{li2016} for new types of entire solutions.

When equation \eqref{1} is used to model the population dynamics of a species,
it is assumed that the underlying environment is not patchy and the internal interaction
of the organisms is nonlocal. Conversely, if we assume that the organisms
move randomly between the adjacent spatial locations, then it is more effective to
use the following classical reaction-diffusion equation
\begin{equation}\label{100}
u_t(t,x)=d\Delta u(t,x)+u(t,x)(r(x-ct)-u(t,x)),
~t>0, x\in\mathbb{R},
\end{equation}
and if the species live in patchy environments, the lattice differential equation of the form
\begin{equation}\label{101}
 u_t(t,x)=d[u(t,x+1)-2u(t,x)+u(t,x-1)]+u(t,x)(r(x-ct)-u(t,x)),
~t>0, x\in\mathbb{Z},
\end{equation}
is more meaningful. Note that equations \eqref{1}-\eqref{101}
are neither homogeneous nor periodic, but possess special heterogeneity
with the form of ``spatial shifting" at a constant speed.
Therefore, we cannot directly apply the abstract theory developed
for monotone semiflows in \cite{LiangZhaoCPAM2007,JFA2010,JFA2017}
to address the issue of spreading speeds and traveling wave solutions.
Certain ad hoc techniques that fit the equation itself are needed and necessary.
Recently, Li and his collaborators \cite{LiBTsiam,LiBTjde}
studied the spatial dynamics of \eqref{101} and \eqref{100}, respectively,
and they showed that the long term behavior of solutions depends on
the speed of the shifting habitat edge $c$ and a number $c^*(\infty)$,
where $c^*(\infty)=2\sqrt{d r(\infty)}$ for \eqref{100}
and $c^*(\infty)=\inf_{\lambda>0}\frac{4d\sinh^2(\lambda/2)+r(\infty)}{\lambda}$ for \eqref{101}.
More accurately, they demonstrated that if $c>c^*(\infty)$,
then the species will become extinct in the habitat,
and if $0<c<c^*(\infty)$, then the species
will persist and spread along the shifting habitat
at the asymptotic spreading speed $c^*(\infty)$.
Very recently, by  the monotone iterative method,
Hu and Zou \cite{HuZou2017} proved that \eqref{100} admits
a monotone traveling wave solution connecting $0$ to $r(\infty)$
with the speed being the habitat
shifting speed, which indeed accounts for an extinction wave.
Here we remark that by a change of variable $v(t,x)=u(t,-x)$,  such a forced traveling wave for \eqref{100}
can also be obtained from \cite[Theorem 2.1(i)]{FangLouWuSIAM}, as applied to the resulting equation.
In addition, Li et al. \cite{LiBMB}
considered this type of problem using an integro-difference equation model
in an expanding or contracting habitat.
In this paper, we propose to extend the above existing results
on equations \eqref{100} and \eqref{101} to our nonlocal dispersal problem \eqref{1}.
To summarize, we first study the persistence and spreading speed properties
by applying the comparison principle
and constructing an appropriate subsolution,
and then establish the existence of traveling wave solutions by constructing super/sub-solution
and using the method of monotone iteration.
We should point out that the combination of nonlocal effects and shifting environment
makes the analysis on model \eqref{1} more difficult.
In particular, the construction of some appropriate subsolutions to study
the spreading speed and traveling waves are highly nontrivial.

The rest of the paper is organized as follows. In Section \ref{pre},
we give some preliminaries including the uniqueness and existence
of solutions and the comparison principle.
Section \ref{ss} is devoted to the persistence and spreading speed.
Finally, we study the traveling wave solutions in Section \ref{tws}.

\section{Preliminaries}\label{pre}
Let
\begin{equation*}
\mathbb{Y}=\{\psi\in C(\mathbb{R},\mathbb{R}): u
\text{ is bounded and uniformly continuous on }\mathbb{R}\}
\end{equation*}
with norm $\|\psi\|=\sup_{x\in\mathbb{R}}|\psi(x)|$,
and $\mathbb{Y}_+=\{\psi\in\mathbb{Y}: \psi(x)\geq0,\forall x\in\mathbb{R}\}$.
It is easily seen that $\mathbb{Y}_+$ is a closed cone of $\mathbb{Y}$
and its induced partial ordering makes $\mathbb{Y}$ into a Banach
lattice. Note that $J\ast u-u: \mathbb{Y}\rightarrow \mathbb{Y}$
is a bounded linear operator with respect to the norm $\|\cdot\|$.
It then follows that the system
\begin{equation}
\begin{cases}
\frac{\partial u(t,x)}{\partial t}=d(J\ast u-u)(t,x),~t>0, x\in\mathbb{R},\\
u(0,x)=\psi(x),~x\in\mathbb{R}, \psi\in \mathbb{Y}
\label{013}
\end{cases}
\end{equation}
generates a strongly continuous semigroup $P(t)$ on $\mathbb{Y}$, which is also strongly
positive in the sense of $P(t)\mathbb{Y}_+\subseteq \mathbb{Y}_+$ and $[P(t)\psi](x)\gg0$
if $\psi(x)\geq 0$ has a nonempty support and $t>0$. According to \cite{WengZhao2006},
the unique mild solution of system \eqref{013} is given by
\begin{equation}
[P(t)\psi](x)=e^{-dt}\sum_{k=0}^{\infty}\frac{(dt)^k}{k!}a_k(\psi)(x),
\label{030}
\end{equation}
where $a_0(\psi)(x)=\psi(x)$ and $a_k(\psi)(x)=\int_{\mathbb{R}}J(x-y)a_{k-1}(\psi)(y){\rm d}y$
for any integer $k\geq1$. On the other hand,
Ignat and Rossi \cite[Section 2]{IgnatRossi2007} showed that the fundamental solution
of \eqref{013} can be decomposed as
\begin{equation}\label{0001}
G(t,x)=e^{-dt}\delta_0(x)+R_t(x),
\end{equation}
where $\delta_0(\cdot)$ is the delta measure at zero and $R_t(x)=R(t,x)$ is smooth defined by
$$R(t,x)=\frac{1}{2\pi}\int_{\mathbb{R}}e^{-dt}(e^{d\widehat{J}(\xi)t}-1)e^{ix\xi}{\rm d}\xi$$
with $i=\sqrt{-1}$ and $\widehat{J}$ being the Fourier transform of $J$.
 Moreover, the solution of \eqref{013} can also be written as
\begin{equation}
\begin{split}
u(t,x)=\int_{\mathbb{R}}G(t,y)\psi(x-y){\rm d}y
=e^{-dt}\psi(x)+\int_{\mathbb{R}}R_t(y)\psi(x-y){\rm d}y, ~t\geq 0, x\in\mathbb{R}.
\label{015}
\end{split}
\end{equation}
It then follows that $u(t,\cdot)$ is as regular as $\psi$ is, and hence
the nonlocal dispersal operator $J\ast u-u$ does not have the regularizing effect
to the Cauchy problem \eqref{013}. Further, we have the following properties about $G(t,x)$.

\begin{lemma}\label{lem2}
$G(t,x)=G(t,-x)$ for all $t\geq0$ and $x\in\mathbb{R}$.
Further, $\int_{\mathbb{R}}G(t,y){\rm d}y=1$
and $\|G(t,\cdot)\|_{L^p(\mathbb{R})}\leq3$ for any $t\geq0$ and $p\in [1,\infty]$.
\end{lemma}
\begin{proof}
By the symmetry of $J$, we have
\begin{equation*}
\widehat{J}(\xi)=\int_{\mathbb{R}}J(x)e^{-ix\xi}{\rm d}x
=\int_{\mathbb{R}}J(x)\cos(x\xi){\rm d}x,
\end{equation*}
which implies that $\widehat{J}(-\xi)=\widehat{J}(\xi)$ and $-1\leq\widehat{J}(\xi)\leq 1$.
Then a direction computation yields that
\begin{align*}
R_t(x)=&\frac{1}{2\pi}\int_{\mathbb{R}}e^{-dt}(e^{d\widehat{J}(\xi)t}-1)e^{ix\xi}{\rm d}\xi\\
=&\frac{1}{2\pi}\int_{\mathbb{R}}e^{-dt}(e^{d\widehat{J}(\xi)t}-1)
[\cos(x\xi)+i\sin(x\xi)]{\rm d}\xi\\
=&\frac{1}{2\pi}\int_{\mathbb{R}}e^{-dt}(e^{d\widehat{J}(\xi)t}-1)\cos(x\xi){\rm d}\xi.
\end{align*}
Therefore, $R_t(-x)=R_t(x)$. By \eqref{0001}, we obtain
\begin{equation*}
  G(t,x)=e^{-dt}\delta_0(x)+\frac{1}{2\pi}\int_{\mathbb{R}}
  e^{-dt}(e^{d\widehat{J}(\xi)t}-1)\cos(x\xi){\rm d}\xi
  =\frac{1}{2\pi}\int_{\mathbb{R}}e^{d(\widehat{J}(\xi)-1)t}\cos(x\xi){\rm d}\xi.
\end{equation*}
Clearly, $G(t,x)=G(t,-x)$ for all $t\geq0$ and $x\in\mathbb{R}$.
Note that $u(t,x)\equiv C$ for $(t,x)\in[0,\infty)\times\mathbb{R}$
is a solution of \eqref{013}, where $C$ is some positive constant.
On the other hand, by \eqref{015}, the solution of \eqref{013}
with initial value $\psi(x)=C$ can be expressed by
$C=\int_{\mathbb{R}}G(t,y)C{\rm d}y$,
which indicates that $\int_{\mathbb{R}}G(t,y){\rm d}y=1$ for all $t\geq0$.
The conclusion of $\|G(t,\cdot)\|_{L^p(\mathbb{R})}\leq3$ with $p\in[1,\infty]$
originates from \cite[Remark 2.1]{IgnatRossi2007}.
\end{proof}

Let $f(x,u)=u(r(x)-u)$. For any $0\leq u_1, u_2\leq r(\infty)$ and
$-\infty<x<\infty$, we can easily verify that
\begin{equation*}
  |f(x,u_1)-f(x,u_2)|\leq (\max\{-r(-\infty),r(\infty)\}+2r(\infty))|u_1-u_2|,
\end{equation*}
which indicates that $f(x,u)$ is Lipschitz continuous in $u\in[0,r(\infty)]$.
Choose $\rho>2r(\infty)-r(-\infty)$, then $\rho u+f(x,u)$ is
nondecreasing in $u\in[0,r(\infty)]$.
Consider the equivalent equation obtained by adding the linear term $\rho u(t,x)$
to both sides of \eqref{1}:
\begin{equation}\label{2}
\frac{\partial u(t,x)}{\partial t}+\rho u(t,x)
=d\left(J\ast u-u\right)(t,x)+u(t,x)(\rho+r(x-ct)-u(t,x)).
\end{equation}
Set $\mathbb{Y}_{r(\infty)}:=\{\psi\in\mathbb{Y}:0\leq\psi(x)\leq r(\infty),
\forall x\in\mathbb{R}\}$. The mild solution of equation \eqref{2} or \eqref{1}
with $u(0,\cdot)=u_0(\cdot)\in \mathbb{Y}_{r(\infty)}$ can be expressed as a fixed point
of the nonlinear integral equation in $C(\mathbb{R}_+,\mathbb{Y}_{r(\infty)})$:
\begin{equation}\label{3}
\begin{split}
u(t,x)=&[\mathcal{N}u](t,x)\\
\triangleq&[e^{-\rho t}P(t)u_0](x)+\int_0^te^{-\rho (t-s)}
P(t-s)u(s,x)(\rho+r(x-cs)-u(s,x)){\rm d}s.
\end{split}
\end{equation}
With the expression of $P(t)$, a direct calculation shows that
\begin{equation*}
\frac{\partial [P(t)\psi](x)}{\partial t}=-d [P(t)\psi](x)
+d\int_{\mathbb{R}}J(y)[P(t)\psi](x-y){\rm d}y,
\end{equation*}
which indicates that the right-side of \eqref{3} is differential with respect to $t$.
Thus, $u(t,x)$ is a classical solution of equation \eqref{2} or \eqref{1}.

\begin{definition}\label{def1}
$u\in C([0,T),\mathbb{Y}_+)$ with $0<T\leq\infty$  is called a
supersolution (subsolution) of \eqref{3} if $u(t,x)\geq(\leq)[\mathcal{N}u](t,x)$
for all $t\in[0,T)$ and $x\in\mathbb{R}$.
\end{definition}

\begin{remark}\label{rem1}
If $u\in C([0,T),\mathbb{Y}_+)$ being $C^1$ in $t\in(0,T)$ satisfies
\eqref{2} or \eqref{1} with $``="$ being replaced by $``\geq"$ ($``\leq"$)
and $u(0,x)\geq (\leq)u_0(x)$, then it follows from the positivity of $P(t)$ that $u$
is a supersolution (subsolution) of \eqref{3}. Moreover, we can easily verify that
$u\equiv r(\infty)$ and $u\equiv 0$ are a supersolution
and a trivial subsolution of \eqref{3}, respectively.
\end{remark}

Now we consider the sequence $\{u^{n}(t,x)\}$ generated by
\begin{equation}\label{4}
  u^{n+1}(t,x)=[\mathcal{N}u^{n}](t,x),
\end{equation}
where $u^{0}(t,x)=0$ or $u^{0}(t,x)=r(\infty)$.

\begin{theorem}\label{thm1}
Let $u_0\in\mathbb{Y}_{r(\infty)}$. Then equation \eqref{3} admits a unique solution
$u\in C(\mathbb{R}_+,\mathbb{Y}_{r(\infty)})$.
Moreover, the comparison principle holds for \eqref{3},
i.e., if $u_1(t,x)$ and $u_2(t,x)$ are two solutions of \eqref{3}
associated with initial value $u_{10}, u_{20}\in \mathbb{Y}_{r(\infty)}$, respectively,
with $u_{10}(x)\leq u_{20}(x)$ for all $x\in\mathbb{R}$, then $u_1(t,x)\leq u_2(t,x)$
for all $t\geq0$ and $x\in\mathbb{R}$. If we further assume that
$u_{10}\not\equiv u_{20}$, then $u_1(t,x)< u_2(t,x)$
for all $t\geq0$ and $x\in\mathbb{R}$.
\end{theorem}
\begin{proof}
This proof is based a classical super-sub solution method
and we only give a sketch here.
Define $\underline{u}^{n+1}(t,x)=[\mathcal{N}\underline{u}^{n}](t,x)$
with $\underline{u}^{0}(t,x)=0$, and $\bar{u}^{n+1}(t,x)=[\mathcal{N}\bar{u}^{n}](t,x)$
with $\bar{u}^{0}(t,x)=r(\infty)$. Then we can show by induction that
\begin{equation*}
  0\leq \underline{u}^{1}(t,x)\leq\cdot\cdot\cdot \leq \underline{u}^{n}(t,x)
  \leq\cdot\cdot\cdot \leq\bar{u}^{n}(t,x)
  \leq\cdot\cdot\cdot \leq \bar{u}^{1}(t,x)\leq r(\infty),
\end{equation*}
which implies that the pointwise limits
\begin{equation*}
  \underline{u}(t,x):=\lim_{n\rightarrow\infty} \underline{u}^{n}(t,x)
  \text{ and } \bar{u}(t,x) :=\lim_{n\rightarrow\infty}\bar{u}^{n}(t,x)
\end{equation*}
both exist and satisfy that $0\leq\underline{u} \leq \bar{u}\leq r(\infty)$.
Moreover, both $\underline{u}$ and $\bar{u}$ are solutions of \eqref{3} in $C(\mathbb{R}_+,\mathbb{Y}_{r(\infty)})$. We now prove $\underline{u}(t,x)=\bar{u}(t,x)$.
Note that for any $\psi\in \mathbb{Y}$, $\|a_0(\psi)\|=\|\psi\|$,
$\|a_1(\psi)\|=\|\int_{\mathbb{R}}J(y)a_{0}(\psi)(\cdot-y){\rm d}y\|\leq\|\psi\|$,
by induction, we can claim that $\|a_k(\psi)\|\leq\|\psi\|$ for all $k=0,1,2,\cdot\cdot\cdot$.
By using \eqref{030}, we obtain
\begin{equation}\label{300}
  \|P(t)\psi\|\leq e^{-dt}\sum_{k=0}^{\infty}\frac{(dt)^k}{k!}\|a_k(\psi)\|\leq\|\psi\|
\text{ for all } t\geq0.
\end{equation}
Therefore, by \eqref{3} and \eqref{300}, a direct calculation yields that
\begin{equation*}
\begin{split}
0\leq \bar{u}(t,x)-\underline{u}(t,x)\leq& (\rho+3r(\infty))
\int_0^te^{-\rho (t-s)}P(t-s)[\bar{u}(s,x)-\underline{u}(s,x)]{\rm d}s\\
\leq&(\rho+3r(\infty))\int_0^te^{-\rho (t-s)}\|P(t-s)
[\bar{u}(s,\cdot)-\underline{u}(s,\cdot)]\|{\rm d}s\\
\leq&(\rho+3r(\infty))
\int_0^te^{-\rho (t-s)}\|\bar{u}(s,\cdot)-\underline{u}(s,\cdot)\|{\rm d}s,
\end{split}
\end{equation*}
which shows that
\begin{equation*}
0\leq e^{\rho t}\|\bar{u}(t,\cdot)-\underline{u}(t,\cdot)\|
\leq (\rho+3r(\infty))\int_0^te^{\rho s}\|\bar{u}(s,\cdot)-\underline{u}(s,\cdot)\|{\rm d}s.
\end{equation*}
Then the Gronwall's inequality implies that
$0\leq e^{\rho t}\|\bar{u}(t,\cdot)-\underline{u}(t,\cdot)\|\leq 0$.
This implies that $\underline{u}(t,x)=\bar{u}(t,x)$ for all $t\geq0$ and $x\in\mathbb{R}$.
The comparison principle is a straightforward consequence of the construction for solutions.
Using the strongly positivity of $P(t)$, we can easily prove the last conclusion of this theorem.
\end{proof}

The following result is a simple consequence  of Theorem \ref{thm1}.
\begin{corollary}\label{cor1}
Let $u,v\in C(\mathbb{R}_+,\mathbb{Y}_{r(\infty)})$ be
the supersolution and subsolution of \eqref{3}
for all $(t,x)\in \mathbb{R}_+\times\mathbb{R}$, respectively.
If $u(0,x)\geq v(0,x)$ for all $x\in\mathbb{R}$,
then $u(t,x)\geq v(t,x)$ for all $(t,x)\in \mathbb{R}_+\times\mathbb{R}$.
\end{corollary}
\begin{proof}
According to the positivity of $P(t)$ and the choice of $\rho$
(i.e., $\rho>2r(\infty)-r(-\infty)$),
which guarantees the monotonicity of $\rho u+f(x-ct,u)$),
we can claim that the nonlinear operator $\mathcal{N}$
defined by \eqref{3} is order preserving in the sense that
\begin{gather*}
  u(t,x)\geq[\mathcal{N}u](t,x)\geq [\mathcal{N}^k u](t,x)
\geq \cdot\cdot\cdot\geq\lim_{k\rightarrow\infty} [\mathcal{N}^k u](t,x)=:\tilde{u}(t,x)\\
  \text{ and }
v(t,x)\leq[\mathcal{N}v](t,x)\leq [\mathcal{N}^k v](t,x)
\leq \cdot\cdot\cdot\leq\lim_{k\rightarrow\infty} [\mathcal{N}^k v](t,x)=:\tilde{v}(t,x).
\end{gather*}
Clearly, $\tilde{u}(0,x)=u(0,x)$ and $\tilde{v}(0,x)=v(0,x)$.
Moreover, both $\tilde{u}(t,x)$ and $\tilde{v}(t,x)$ are the solutions of \eqref{3},
and hence, Theorem \ref{thm1} implies that
$u(t,x)\geq \tilde{u}(t,x)\geq \tilde{v}(t,x)\geq v(t,x)$
because $\tilde{u}(0,x)\geq\tilde{v}(0,x)$, which derive the requested result.
\end{proof}

\section{Persistence and spreading speeds}\label{ss}

In this section, we first show that the species will become extinct
in the long run if the edge of the habitat shifts relatively fast.
For $r(x)>0$ and $\lambda>0$, we define
\begin{equation*}
  \phi(x;\lambda)
  =\frac{d\left(\int_{\mathbb{R}}J(y)e^{\lambda y}{\rm d}y-1\right)+r(x)}{\lambda}.
\end{equation*}
Clearly, $\phi(x;\lambda)>0$ and $\phi(x;\lambda)\rightarrow\infty$ as $\lambda\rightarrow0$.
On the other hand,
\begin{equation*}
\begin{split}
\phi(x;\lambda)=&\frac{d\left(\int_{\mathbb{R}}J(y)\Sigma_{n=0}^\infty
\frac{(\lambda y)^n}{n!}{\rm d}y-1\right)+r(x)}{\lambda}\\
=&\frac{d\int_{\mathbb{R}}J(y)\Sigma_{m=1}^\infty
\frac{(\lambda y)^{2m}}{(2m)!}{\rm d}y+r(x)}{\lambda}\\
=& d\Sigma_{m=1}^\infty\frac{\lambda^{2m-1}\int_{\mathbb{R}}J(y)y^{2m}{\rm d}y}{(2m)!}
+\frac{r(x)}{\lambda}\rightarrow\infty \text{ as } \lambda\rightarrow\infty
\end{split}
\end{equation*}
and
\begin{equation*}
  \frac{\partial^2 \phi(x;\lambda)}{\partial \lambda^2}
=d\Sigma_{m=1}^\infty\frac{(2m-1)(2m-2)\lambda^{2m-3}\int_{\mathbb{R}}J(y)y^{2m}{\rm d}y}{(2m)!}
+\frac{2r(x)}{\lambda^3}>0.
\end{equation*}
Then we can further check that for some fixed $x$, $\phi(x;\lambda)$ has only one minimum denoted by $c^*(x)$, i.e.,
\begin{equation*}
  c^*(x)=\min_{\lambda>0}\phi(x;\lambda)=\phi(x;\lambda^*(x))>0,
\end{equation*}
where $\lambda^*(x)$ denotes the unique point where the minimum occurs.

\begin{theorem}\label{thm2}
Assume that $c>c^*(\infty)\triangleq\min_{\lambda>0}
\frac{d\left(\int_{\mathbb{R}}J(y)e^{\lambda y}{\rm d}y-1\right)+r(\infty)}{\lambda}$.
If $u_0\in\mathbb{Y}_{r(\infty)}$ and $u_0(x)\equiv0$ for all sufficiently large $|x|$,
then for every $\epsilon>0$ there exists $T>0$ such that for any $t\geq T$,
the solution of \eqref{1} with $u(0,x)=u_0(x)$ satisfies $u(t,x)<\epsilon$
for all $x\in\mathbb{R}$.
\end{theorem}
\begin{proof}
Choose sufficiently large $M>0$ such that $r(-M)<0$.
This can be done due to the continuity of $r(\cdot)$
and $r(-\infty)<0$. From Theorem \ref{thm1},
we have $0\leq u(t,x)\leq r(\infty)$ since $u(0,x)=u_0(x)\in \mathbb{Y}_{r(\infty)}$.
It then follows that for $x-ct\leq-M$,
\begin{equation*}
\begin{split}
\frac{\partial u(t,x)}{\partial t}=&d\left(J\ast u-u\right)(t,x)+u(t,x)(r(x-ct)-u(t,x))\\
\leq&  d\left(J\ast u-u\right)(t,x)+u(t,x)(r(-M)-u(t,x))\\
\leq&  d\left(J\ast u-u\right)(t,x)+r(-M)u(t,x).
\end{split}
\end{equation*}
Note that $\hat{u}(t,x)=a e^{r(-M)t}$ is a solution of the following linear equation
$$\hat{u}_t(t,x)=d\left(J\ast \hat{u}-\hat{u}\right)(t,x)+r(-M)\hat{u}(t,x),$$
where $a$ is a positive constant such that $a\geq u_0$.
Therefore, $\hat{u}(t,x)=a e^{r(-M)t}$ is a supersolution of \eqref{1}
on the domain $\{(t,x)\in\mathbb{R}_+\times\mathbb{R}:x-ct\leq-M\}$
and the comparison principle implies that
\begin{equation*}
  0\leq u(t,x)\leq a e^{r(-M)t}, ~\forall t\geq 0, x-ct\leq-M,
\end{equation*}
which, together with $r(-M)<0$, indicates that for any $\epsilon>0$,
there exists $T_1>0$ such that
\begin{equation}\label{10}
  u(t,x)<\epsilon, ~\forall t\geq T_1, x\leq-M+ct.
\end{equation}

Pick $\delta\in (0,c-c^*(\infty))$ and let $\lambda_\delta>0$
be the smaller positive solution of $\phi(\infty;\lambda)=c^*(\infty)+\delta/2$,
that is, $$d\left(\int_{\mathbb{R}}J(y)e^{\lambda_\delta y}{\rm d}y-1\right)+r(\infty)
=\lambda_\delta\left(c^*(\infty)+\frac{\delta}{2}\right).$$
Note that $\bar{u}(t,x)=Ae^{-\lambda_\delta(x-(c^*(\infty)+\delta/2)t)}$
with $A$ being a positive constant is a solution of
the following linear nonlocal dispersal equation:
\begin{equation*}
  \bar{u}_t(t,x)=d(J\ast \bar{u}-\bar{u})(t,x)+r(\infty)\bar{u}(t,x),
\end{equation*}
which, together with the fact that $r(\infty)u(t,x)\geq u(t,x)(r(x-ct)-u(t,x))$,
shows that
$\bar{u}(t,x)$ is a supersolution of \eqref{1}. Since $u_0\in\mathbb{Y}_{r(\infty)}$
and $u_0(x)\equiv0$ for all sufficiently large $|x|$, we can choose $A$ large enough
such that $u_0(x)\leq\bar{u}(0,x)=Ae^{-\lambda_\delta x}$,
then the comparison principle yields that
\begin{equation*}
  u(t,x)\leq Ae^{-\lambda_\delta(x-(c^*(\infty)+\delta/2)t)}
  =Ae^{-\lambda_\delta(x-(c^*(\infty)+\delta)t)}e^{-\frac{\lambda_\delta \delta }{2}t}.
\end{equation*}
This implies that $u(t,x)\leq Ae^{-\frac{\lambda_\delta \delta }{2}t}$
for all $x\geq(c^*(\infty)+\delta)t$. Thus, for the above given $\epsilon>0$,
there exists $T_2>0$ such that
\begin{equation}\label{11}
u(t,x)<\epsilon, ~\forall t\geq T_2, x\geq(c^*(\infty)+\delta)t.
\end{equation}
By the choice of $\delta$, i.e., $c>c^*(\infty)+\delta$,
we can further find $T_3>0$ such that for $t\geq T_3$,
there holds that $-M+ct\geq (c^*(\infty)+\delta)t$.
In view of \eqref{10} and \eqref{11},
we can conclude that $u(t,x)<\epsilon$ for all $t\geq T=:\max\{T_1,T_2,T_3\}$
and all $x\in\mathbb{R}$.
\end{proof}

In the rest of this section, we consider the case that the edge of habitat
is moving at a speed less than $c^*(\infty)$.
The construction of a suitable subsolution plays a key role in our theoretical analysis.
To proceed, we introduce an auxiliary function that can be found
in Weinberger's pioneering work \cite{weinbergerSIAM1982},
see also \cite{LiBTsiam,LiBTjde}.
For $\lambda>0$ and $\gamma>0$, define
\begin{equation*}
\upsilon(x;\lambda,\gamma)=\left\{
                               \begin{array}{ll}
                                 e^{-\lambda x}\sin(\gamma x),
& \text{if } 0\leq x\leq\pi/\gamma, \\
                                 0, & \text{elsewhere}.
                               \end{array}
                             \right.
\end{equation*}
Note that $\upsilon(x;\lambda,\gamma)$ is nonnegative, continuous in $x\in\mathbb{R}$
and continuously differentiable when $x\neq 0, \pi/\gamma$.
Moreover, $\upsilon(x;\lambda,\gamma)$ takes the maximum at the point
$x=\sigma(\lambda,\gamma):=(1/\gamma)\arctan(\gamma/\lambda)\in(0,\pi/\gamma)$.
Note that $\sigma(\lambda,\gamma)$ is strictly decreasing in $\lambda>0$
and also that the maximum $\upsilon(\sigma(\lambda,\gamma);\lambda,\gamma)\in (0,1)$.

Define
\begin{equation}\label{14}
  \varphi(\lambda,\gamma)=d\int_{\mathbb{R}}J(y)e^{\lambda y}
\frac{\sin(\gamma y)}{\gamma}{\rm d}y.
\end{equation}
Assume that the compact support of the kernel $J$ is $[-L,L]$.
Since $J$ is symmetric, we have
$$\varphi(\lambda,\gamma)=d\int_0^L J(y)
\left(e^{\lambda y}-e^{-\lambda y}\right)\frac{\sin(\gamma y)}{\gamma}{\rm d}y.$$
From now on, $\gamma>0$ will be assumed to be sufficiently small
so that both $\sin(\gamma y)>0$ and $\cos(\gamma y)>0$ for $y\in(0,L]$.
Particularly, we can choose $\gamma\in(0,\frac{\pi}{2L})$.
Obviously, $\varphi(\lambda,\gamma)>0$.
Further, a direct computation leads to
\begin{equation*}
  \frac{\partial \varphi(\lambda,\gamma) }{\partial \lambda}
=\frac{d\int_0^L y J(y)(e^{\lambda y}+e^{-\lambda y})\sin(\gamma y){\rm d}y}{\gamma}>0,
\end{equation*}
which indicates that $\varphi(\lambda,\gamma)$ is increasing in $\lambda>0$.
Let
\begin{equation}\label{17}
  \phi_\gamma(l;\lambda)
:=\frac{d\left(\int_{\mathbb{R}}
J(y)e^{\lambda y}\cos(\gamma y){\rm d}y-1\right)+r(l)}{\lambda}
\end{equation}
and
\begin{equation*}
  c^*_\gamma(l)=\min_{\lambda>0}\phi_\gamma(l;\lambda).
\end{equation*}
Clearly, $\phi_\gamma(l;\lambda)<\phi(l;\lambda)$
and $\phi_\gamma(l;\lambda)$ converges to $\phi(l;\lambda)$
uniformly for $\lambda$ in any bounded interval as $\gamma\rightarrow0$.
Moreover, we have $c^*_\gamma(l)<c^*(l)$ and the convergence
$c^*_\gamma(l)\rightarrow c^*(l)$ as $\gamma\rightarrow 0$.

\begin{lemma}\label{lem3}
Assume that $c\in(0,c^*(\infty))$. For any $0<\delta<\frac{c^*(\infty)-c}{5}$,
let $l$ be the point such that $c^*(l)=c^*(\infty)-\delta$,
and small $\gamma\in(0,\frac{\pi}{2L})$ such that $c^*(l)-c^*_\gamma(l)\leq \delta$.
Choose $0<\lambda_1< \lambda_2<\lambda^*(l)$ satisfying
$\varphi(\lambda_1,\gamma)=c+\delta$ and $\varphi(\lambda_2,\gamma)=c^*_\gamma(l)-2\delta$.
Then for any $\lambda\in [\lambda_1,\lambda_2]$ and small $a>0$,
$w(t,x)=a\upsilon(x-l-\varphi(\lambda,\gamma)t)$ is a continuous subsolution of \eqref{2}.
Furthermore, if $u_0(x)\geq a\upsilon(x-l;\lambda,\gamma)$,
then $u(t,x)\geq a\upsilon(x-l-\varphi(\lambda,\gamma)t;\lambda,\gamma)$
for all $t>0$ and $x\in\mathbb{R}$, where $u(t,x)$ is the solution of \eqref{2}
with $u(0,x)=u_0(x)$.
\end{lemma}
\begin{proof}
By Definition \ref{def1}, we need to justify that $w(t,x)\leq [\mathcal{N}w](t,x)$
for all $(t,x)\in \mathbb{R}_+\times\mathbb{R}$. Notice that for $t>0$
and $x<l+\varphi(\lambda,\gamma)t$ or $x>l+\varphi(\lambda,\gamma)t+\pi/\gamma$,
$w(t,x)\equiv 0$, then the proof is trivial. Now we consider the case where $t>0$
and $l+\varphi(\lambda,\gamma)t\leq x\leq l+\varphi(\lambda,\gamma)t+\pi/\gamma$.
At present,
\begin{equation*}
  w(t,x)=a\upsilon(x-l-\varphi(\lambda,\gamma)t;\lambda,\gamma)
  =a e^{-\lambda(x-l-\varphi(\lambda,\gamma)t)}\sin[\gamma(x-l-\varphi(\lambda,\gamma)t)].
\end{equation*}
Clearly, $w(t,x)$ is continuously differential with respect to $t$ in such case.
According to Remark \ref{rem1}, it suffices to prove that for $t>0$
and $l+\varphi(\lambda,\gamma)t\leq x\leq l+\varphi(\lambda,\gamma)t+\pi/\gamma$,
there holds that
\begin{equation}\label{12}
  \frac{\partial w(t,x)}{\partial t}
\leq d\left(\int_{\mathbb{R}}J(y)w(t,x-y){\rm d}y-w(t,x) \right)+w(t,x)(r(x-ct)-w(t,x)).
\end{equation}

By a direct calculation, we obtain
\begin{equation*}
  \frac{\partial w(t,x)}{\partial t}
  =a \varphi(\lambda,\gamma) e^{-\lambda(x-l-\varphi(\lambda,\gamma)t)}
  \bigg(\lambda\sin[\gamma(x-l-\varphi(\lambda,\gamma)t)]
  -\gamma\cos[\gamma(x-l-\varphi(\lambda,\gamma)t)]\bigg).
\end{equation*}
Note that if $y\in supp(J)=[-L,L]$, we have
$$l+\varphi(\lambda,\gamma)t-L\leq x-y\leq l+\varphi(\lambda,\gamma)t+\pi/\gamma+L.$$
It then follows that $$ -\frac{\pi}{2}\leq-\gamma L\leq\gamma(x-y-l-\varphi(\lambda,\gamma)t)
\leq \pi+\gamma L \leq\frac{3\pi}{2}.$$
For such $t$ and $x$, there holds that $$w(t,x-y)\geq
e^{-\lambda(x-y-l-\varphi(\lambda,\gamma)t)}\sin[\gamma(x-y-l-\varphi(\lambda,\gamma)t)].$$
Therefore,
\begin{align*}
  &d\left(\int_{\mathbb{R}}J(y)w(t,x-y){\rm d}y-w(t,x) \right)
=d\left(\int_{-L}^{L}J(y)w(t,x-y){\rm d}y-w(t,x) \right)\\
  &\geq ad \bigg(\int_{-L}^{L}J(y)e^{\lambda y}
e^{-\lambda(x-l-\varphi(\lambda,\gamma)t)}
\sin[\gamma(x-y-l-\varphi(\lambda,\gamma)t)]{\rm d}y\\
  &\ \ \ \ -e^{-\lambda(x-l-\varphi(\lambda,\gamma)t)}
  \sin[\gamma(x-l-\varphi(\lambda,\gamma)t)] \bigg)\\
  &=ad e^{-\lambda(x-l-\varphi(\lambda,\gamma)t)}
\bigg(\int_{\mathbb{R}}J(y)e^{\lambda y}\cos( \gamma y){\rm d}y
\sin[\gamma(x-l-\varphi(\lambda,\gamma)t)]\\
  &\ \ \ \ -\int_{\mathbb{R}}J(y)e^{\lambda y}\sin(\gamma y){\rm d}y
\cos[\gamma(x-l-\varphi(\lambda,\gamma)t)]-\sin[\gamma(x-l-\varphi(\lambda,\gamma)t)] \bigg).
\end{align*}
To prove claim \eqref{12}, one only need to show
\begin{equation}\label{13}
\begin{split}
  &\lambda \varphi(\lambda,\gamma)\sin[\gamma(x-l-\varphi(\lambda,\gamma)t)]
  \\
  &\leq d\left(\int_{\mathbb{R}}J(y)e^{\lambda y}\cos( \gamma y){\rm d}y -1\right)
\sin[\gamma(x-l-\varphi(\lambda,\gamma)t)]\\
  &\ \ \ \ +\left(\gamma \varphi(\lambda,\gamma)-d\int_{\mathbb{R}}J(y)e^{\lambda y}
\sin( \gamma y){\rm d}y \right) \cos[\gamma(x-l-\varphi(\lambda,\gamma)t)]\\
  &\ \ \ \ +\sin[\gamma(x-l-\varphi(\lambda,\gamma)t)]
  \bigg(r(x-ct)-a\upsilon(x-l-\varphi(\lambda,\gamma)t;\lambda,\gamma)\bigg),
\end{split}
\end{equation}
which is equivalent to
\begin{equation}\label{15}
\lambda \varphi(\lambda,\gamma)\leq d\left(\int_{\mathbb{R}}J(y)e^{\lambda y}
\cos( \gamma y){\rm d}y -1\right)
+ r(x-ct)-a\upsilon(x-l-\varphi(\lambda,\gamma)t;\lambda,\gamma)
\end{equation}
due to \eqref{14} and $\sin[\gamma(x-l-\varphi(\lambda,\gamma)t)]>0$
for $l+\varphi(\lambda,\gamma)t< x<l+\varphi(\lambda,\gamma)t+\pi/\gamma$.
Note that \eqref{13} holds naturally when $x=l+\varphi(\lambda,\gamma)t$
or $l+\varphi(\lambda,\gamma)t+\pi/\gamma$.

According to the hypothesis,
$\varphi(\lambda_1,\gamma)=c+\delta$ and $\varphi(\lambda_2,\gamma)=c^*_\gamma(l)-2\delta$
for some $0<\lambda_1< \lambda_2<\lambda^*(l)$. Then for $\lambda\in[\lambda_1,\lambda_2]$,
$x>l+\varphi(\lambda,\gamma)t\geq l+\varphi(\lambda_1,\gamma)t=l+(c+\delta)t>l+ct$ with $t>0$,
that is, $x-ct>l$. Since $r(\cdot)$ is nondecreasing, then $r(x-ct)\geq r(l)$.
Notice also that $\upsilon(x-l-\varphi(\lambda,\gamma)t;\lambda,\gamma)\leq1$.
Therefore, in order to prove \eqref{15}, we only need to verify
\begin{equation}\label{16}
\lambda \varphi(\lambda,\gamma)
\leq d\left(\int_{\mathbb{R}}J(y)e^{\lambda y}\cos( \gamma y){\rm d}y -1\right) + r(l)-a,
\end{equation}
which, together with \eqref{17}, indicates that
\begin{equation}\label{18}
a\leq\lambda(\phi_\gamma(l;\lambda)-\varphi(\lambda,\gamma)).
\end{equation}
Based on the previous parameter setting,
we have $$\phi_\gamma(l;\lambda)-\varphi(\lambda,\gamma)
\geq c^*_\gamma(l)-\varphi(\lambda_2,\gamma)
=c^*_\gamma(l)-(c^*_\gamma(l)-2\delta)=2\delta.$$
Thus, \eqref{18} holds as long as we select $0<a\leq 2\lambda_1\delta$,
which, in return, shows that for $\lambda\in[\lambda_1,\lambda_2]$
and sufficiently small $a>0$,
$w(t,x)=a\upsilon(x-l-\varphi(\lambda,\gamma)t;\lambda,\gamma)$
is a continuous subsolution of \eqref{2}.
Furthermore, if $u_0(x)\geq a\upsilon(x-l;\lambda,\gamma)$,
then it follows from Corollary \ref{cor1}
that $u(t,x)\geq a\upsilon(x-l-\varphi(\lambda,\gamma)t;\lambda,\gamma)$
for all $t>0$ and $x\in\mathbb{R}$.
\end{proof}

Now we are ready to prove the main result of this section, which implies that
if the edge of the habitat suitable for species growth
is shifting at a speed $c<c^*(\infty)$,
then the species will persists in the space and spread to the right
at the asymptotic speed $c^*(\infty)$.

\begin{theorem}\label{persist-ss}
Assume that $c^*(\infty)>c>0$. Let $u(t,x)$ be the solution of \eqref{1}
with $u(0,\cdot)=u_0(\cdot)\in \mathbb{Y}_{r(\infty)}$.
Then the following statements are valid:
\begin{description}
  \item[(i)] for any $\varsigma>0$,
$$\lim_{t\rightarrow\infty}\sup_{x\leq(c-\varsigma)t} u(t,x)=0.$$
  \item[(ii)] if $u_0(x)\equiv0$ for sufficiently large $|x|$,
then for any $\varsigma>0$,
     $$\lim_{t\rightarrow\infty}\sup_{x\geq(c^*(\infty)+\varsigma)t} u(t,x)=0.$$
  \item[(iii)] if $u_0(x)>0$ on a closed interval,
then for each $\varsigma\in\left(0,\frac{c^*(\infty)-c}{2}\right)$,
there holds
      $$\lim_{t\rightarrow\infty,~(c+\varsigma)t
\leq x\leq(c^*(\infty)-\varsigma)t} u(t,x)=r(\infty).$$
\end{description}
\end{theorem}
\begin{proof}
(i) By Theorem \ref{thm1}, we have $0\leq u(t,x)\leq r(\infty)$ for all $t\geq0$
and $x\in\mathbb{R}$. The first part of the  proof of Theorem \ref{thm2}
shows that for any $\epsilon>0$, there exists large numbers $M>0$ and $T_1>0$ such that
\begin{equation*}
u(t,x)<\epsilon, ~\forall t\geq T_1, x\leq-M+ct.
\end{equation*}
On the other hand, for any given $\varsigma>0$, there exists $T_2>0$ such that
for $t\geq T_2$, $(c-\varsigma)t\leq -M+ct$.
In fact, this can be done if we let $T_2\geq M/\varsigma$.
Thus, we can obtain that
\begin{equation*}
  u(t,x)<\epsilon, ~\forall t\geq \{T_1,T_2\}, x\leq (c-\varsigma)t.
\end{equation*}
This completes the proof of (i).

(ii) For any $\varsigma>0$, let $\lambda_\varsigma>0$
be the smaller positive solution of $\phi(\infty;\lambda)=c^*(\infty)+\varsigma/2$.
In other words, we have
$d\left(\int_{\mathbb{R}}J(y)e^{\lambda_\varsigma y}{\rm d}y-1\right)+r(\infty)
=\lambda_\varsigma(c^*(\infty)+\varsigma/2)$.
Note that $\hat{u}(t,x)=Ae^{-\lambda_\varsigma(x-(c^*(\infty)+\varsigma/2)t)}$
with $A>0$ being a constant satisfies
\begin{equation*}
  \hat{u}_t(t,x)=d(J\ast \hat{u}-\hat{u})(t,x)+r(\infty)\hat{u}(t,x).
\end{equation*}
This, together with the fact that $r(\infty)u(t,x)\geq u(t,x)(r(x-ct)-u(t,x))$,
shows that
$\hat{u}(t,x)$ is a supersolution of \eqref{1}. Since $u_0\in\mathbb{Y}_{r(\infty)}$
and $u_0(x)\equiv0$ for all sufficiently large $|x|$,
we can choose $A$ large enough such that $u_0(x)\leq\hat{u}(0,x)=Ae^{-\lambda_\varsigma x}$,
and hence, the comparison principle yields that
\begin{equation*}
  0\leq u(t,x)\leq Ae^{-\lambda_\varsigma(x-(c^*(\infty)+\varsigma/2)t)}
  =Ae^{-\lambda_\varsigma(x-(c^*(\infty)+\varsigma)t)}
e^{-\frac{\lambda_\varsigma \varsigma }{2}t}.
\end{equation*}
This implies that $0\leq u(t,x)\leq Ae^{-\frac{\lambda_\varsigma \varsigma }{2}t}$
for all $x\geq(c^*(\infty)+\varsigma)t$.
Therefore, we can conclude that
$\lim_{t\rightarrow\infty}\sup_{x\geq(c^*(\infty)+\varsigma)t} u(t,x)=0.$

(iii) Choose $\delta$ small enough with
$0<\delta<\min\left\{\frac{r(\infty)}{\lambda^*(\infty)},\frac{c^*(\infty)-c}{5}\right\}$
and let $l$, $\lambda_1$, $\lambda_2$ and $\gamma$ be as in Lemma \ref{lem3}.
Then Lemma \ref{lem3} implies that for any $\lambda\in[\lambda_1,\lambda_2]$
and small $\alpha>0$, $\frac{\alpha}{\upsilon(\sigma(\lambda,\gamma);\lambda,\gamma)}
\upsilon(x-l-\varphi(\lambda,\gamma)t;\lambda,\gamma)$
is a continuous subsolution of \eqref{1}.

Since $u_0\in\mathbb{Y}_{r(\infty)}$ and $u_0(x)>0$ on a closed interval,
it follows from Theorem \ref{thm1} that $u(t,x)>0$ for all $t>0$ and $x\in\mathbb{R}$.
Choose $0<t_0\leq \frac{\sigma(\lambda_1,\gamma)}{c}$ with $c>0$,
$\alpha>0$ and $\gamma>0$ sufficiently small such that
$u(t_0,x)\geq \alpha$ for $x\in[l,l+4\pi/\gamma]$.
Define
\begin{equation*}
  W(0,x)=\left\{
           \begin{array}{ll}
             \frac{\alpha}{\upsilon(\sigma(\lambda_1,\gamma);\lambda_1,\gamma)}
             \upsilon(x-l;\lambda_1,\gamma),
& \text{ if } l \leq x\leq l+\sigma(\lambda_1,\gamma),  \\
             \alpha, & \text{ if } l+\sigma(\lambda_1,\gamma)\leq x\leq l+\frac{3\pi}{\gamma}
             +\sigma(\lambda_2,\gamma), \\
             \frac{\alpha}{\upsilon(\sigma(\lambda_2,\gamma);\lambda_2,\gamma)}
             \upsilon(x-l-\frac{3\pi}{\gamma};\lambda_2,\gamma),
& \text{ if } l+\frac{3\pi}{\gamma}+\sigma(\lambda_2,\gamma)
             \leq x\leq l+\frac{4\pi}{\gamma}, \\
             0, & \text{ eslewhere. }
           \end{array}
         \right.
\end{equation*}
By the definition of $\upsilon$, we can easily check that for
$\varrho\in[0,2\pi/\gamma]$, there holds
\begin{align*}
  W(0,x)\geq&  \frac{\alpha}{\upsilon(\sigma(\lambda_1,\gamma);\lambda_1,\gamma)}
             \upsilon(x-l-\varrho;\lambda_1,\gamma),\\
  W(0,x)\geq&\frac{\alpha}{\upsilon(\sigma(\lambda_2,\gamma);\lambda_2,\gamma)}
             \upsilon(x-l-\frac{3\pi}{\gamma}+\varrho;\lambda_2,\gamma).
\end{align*}
Clearly, $u(t_0,x)\geq \alpha \geq W(0,x)$ for $x\in[l,l+4\pi/\gamma]$.
Applying Lemma \ref{lem3}, we have
\begin{align}
  u(t,x)\geq& \frac{\alpha}{\upsilon(\sigma(\lambda_1,\gamma);\lambda_1,\gamma)}
             \upsilon(x-l-\varphi(\lambda_1,\gamma)(t-t_0)-\varrho;\lambda_1,\gamma),
              \label{19}\\
  u(t,x)\geq& \frac{\alpha}{\upsilon(\sigma(\lambda_1,\gamma);\lambda_1,\gamma)}
             \upsilon(x-l-\frac{3\pi}{\gamma}-
             \varphi(\lambda_1,\gamma)(t-t_0)+\varrho;\lambda_1,\gamma), \label{20}
\end{align}
for $t\geq t_0$ and $0\leq \varrho \leq 2\pi/\gamma$.
Moreover, by the arguments similar to those in \cite[Theorem 2.2 (iii)]{LiBTsiam}, we can show  that for all $t\geq t_0$,
\begin{equation}\label{21}
  u(t,x)\geq W(t-t_0,x),
\end{equation}
where
\begin{equation}\label{22}
\begin{split}
  &W(t-t_0,x)\\
&=\left\{
               \begin{array}{ll}
                 \frac{\alpha}{\upsilon(\sigma(\lambda_1,\gamma);\lambda_1,\gamma)}
              & \text{ if } l+\varphi(\lambda_1,\gamma)(t-t_0)\leq x\\
                 \times\upsilon(x-l-\varphi(\lambda_1,\gamma)(t-t_0);\lambda_1,\gamma),
              & \quad \leq l+\varphi(\lambda_1,\gamma)(t-t_0)+\sigma(\lambda_1,\gamma), \\
                 \alpha,
& \text{ if } l+\varphi(\lambda_1,\gamma)(t-t_0)+\sigma(\lambda_1,\gamma) \leq x  \\
                  &\quad \leq l+\varphi(\lambda_2,\gamma)(t-t_0)
                   +\sigma(\lambda_2,\gamma)+\frac{3\pi}{\gamma},\\
                 \frac{\alpha}{\upsilon(\sigma(\lambda_2,\gamma);\lambda_2,\gamma)}
               &\text{ if } l+\varphi(\lambda_2,\gamma)(t-t_0)+\sigma(\lambda_2,\gamma )
+\frac{3\pi}{\gamma},\\
                  \times\upsilon(x-l-\frac{3\pi}{\gamma}
                -\varphi(\lambda_2,\gamma)(t-t_0);\lambda_2,\gamma),
             & \quad \leq x\leq l+\frac{4\pi}{\gamma}+\varphi(\lambda_2,\gamma)(t-t_0) \\
                 0, & \text{ eslewhere. }
               \end{array}
             \right.
\end{split}
\end{equation}

Choose sufficiently large $t_1>t_0$ as the initial time.
It then follows from \eqref{3} that for any $t>t_1$, $u(t,x)$ satisfies
\begin{equation}\label{23}
\begin{split}
u(t,x)=&[e^{-\rho (t-t_1)}P(t-t_1)u(t_1,\cdot)](x)\\
&+\int_{t_1}^t\left[e^{-\rho (t-s)}P(t-s)
u(s,\cdot)(\rho+r(\cdot-cs)-u(s,\cdot))\right](x){\rm d}s.\\
\end{split}
\end{equation}
By \eqref{21}, the nondecreasing monotonicity of $u(\rho+r(x-ct)-u)$ on $u$
and the positivity of $P(t)$, we further get that
\begin{equation}\label{24}
\begin{split}
u(t,x)\geq& [e^{-\rho (t-t_1)}P(t-t_1)W(t_1-t_0,\cdot)](x)\\
&+\int_{t_1}^t\left[e^{-\rho (t-s)}P(t-s)W(s-t_0,\cdot)
(\rho+r(\cdot-cs)-W(s-t_0,\cdot))\right](x){\rm d}s,\\
\end{split}
\end{equation}
where $t>t_1$. By the definition of $P(t)$ (see \eqref{030}),
for the linear part, we obtain
\begin{equation*}
\begin{split}
&[e^{-\rho (t-t_1)}P(t-t_1)W(t_1-t_0,\cdot)](x)\\
&=e^{-\rho (t-t_1)}e^{-d(t-t_1)}
\sum_{k=0}^{\infty}\frac{[d(t-t_1)]^k}{k!}J^{(k)}\ast W(t_1-t_0,x)\\
&\geq e^{-\rho (t-t_1)}e^{-d(t-t_1)}
\sum_{k=0}^{N}\frac{[d(t-t_1)]^k}{k!}J^{(k)}\ast W(t_1-t_0,x)\\
&=e^{-\rho (t-t_1)}e^{-d(t-t_1)}\bigg(W(t_1-t_0,x)+\frac{d(t-t_1)}{1!}
\int_{-L}^{L}J(x_1)W(t_1-t_0,x-x_1){\rm d}x_1\\
&\ \ \ \ +\frac{[d(t-t_1)]^2}{2!}\int_{-L}^{L}\int_{-L}^{L}J(x_1)J(x_2)
W(t_1-t_0,x-x_1-x_2){\rm d}x_1{\rm d}x_2+\cdot\cdot\cdot\\
&\ \ \ \ +\frac{[d(t-t_1)]^N}{N!}\underbrace{\int_{-L}^{L}\cdot\cdot\cdot
\int_{-L}^{L}}_{N \text{ terms }}\prod_{i=1}^{N}J(x_i)
W\left(t_1-t_0,x-\sum_{i=1}^{N}x_i\right)
{\rm d}x_1{\rm d}x_2\cdot\cdot\cdot{\rm d}x_N\bigg),
\end{split}
\end{equation*}
where $J^{(0)}=\delta_0$ and $J^{(k)}\ast=J\ast J^{(k-1)}\ast$ for any $k\geq 1$, $N$ is some positive integer, and $[-L,L]:=supp(J)$. Note that when $t\geq t_1$ and $x$ satisfies
\begin{equation}\label{25}
 l+\varphi(\lambda_1,\gamma)(t-t_0)+\sigma(\lambda_1,\gamma) +NL \leq x
                 \leq l+\varphi(\lambda_2,\gamma)(t-t_0)
                   +\sigma(\lambda_2,\gamma)+\frac{3\pi}{\gamma}-NL,
\end{equation}
which does make sense by choosing $t_1>t_0+\frac{\sigma(\lambda_1,\gamma)-\sigma(\lambda_2,\gamma)+2NL}
{\varphi(\lambda_2,\gamma)-\varphi(\lambda_1,\gamma)}$, we see from \eqref{22} that
\begin{equation*}
  W(t_1-t_0,x)\geq \alpha,
  ~W\left(t_1-t_0,x-\sum_{i=1}^{\tilde{N}}x_i\right)\geq \alpha \text{ for } x_i\in[-L,L]
\text{ and } \tilde{N}=1,2,\cdot\cdot\cdot,N.
\end{equation*}
This, together with the fact that  $\int_{-L}^{L}J(y){\rm d}y=1$,
implies that for any $\epsilon>0$, there exists sufficiently large $N_1>0$
such that for $N\geq N_1$,
\begin{equation}\label{26}
\begin{split}
&[e^{-\rho (t-t_1)}P(t-t_1)W(t_1-t_0,\cdot)](x)\\
&\geq \alpha e^{-\rho (t-t_1)}e^{-d(t-t_1)}\sum_{k=0}^{N}\frac{[d(t-t_1)]^k}{k!}\\
&=\alpha e^{-\rho (t-t_1)}e^{-d(t-t_1)}\left(e^{d(t-t_1)}
-\sum_{k=N+1}^{\infty}\frac{[d(t-t_1)]^k}{k!}\right)\\
&=\alpha e^{-\rho (t-t_1)}\left(1-e^{-d(t-t_1)}
\sum_{k=N+1}^{\infty}\frac{[d(t-t_1)]^k}{k!}\right)\\
&\geq\alpha (1-\epsilon)e^{-\rho (t-t_1)}.
\end{split}
\end{equation}
Regarding the nonlinear part, for any $s\in(t_1,t)$, we have
\begin{equation*}
\begin{split}
&\left[e^{-\rho (t-s)}P(t-s)W(s-t_0,\cdot)(\rho+r(\cdot-cs)-W(s-t_0,\cdot))\right](x)\\
&\geq e^{-\rho (t-s)}e^{-d(t-s)}
\sum_{k=0}^{N}\frac{[d(t-s)]^k}{k!}J^{(k)}\ast W(s-t_0,x)(\rho+r(x-cs)-W(s-t_0,x))\\
&=e^{-\rho (t-s)}e^{-d(t-s)}\bigg(W(s-t_0,x)(\rho+r(x-cs)-W(s-t_0,x))\\
&\ \ \ \ +\frac{[d(t-s)]}{1!}\int_{-L}^{L}J(x_1)
W(s-t_0,x-x_1)(\rho+r(x-x_1-cs)-W(s-t_0,x-x_1)){\rm d}x_1 \\
&\ \ \ \ +\cdot\cdot\cdot
+\frac{[d(t-s)]^N}{N!}\underbrace{\int_{-L}^{L}\cdot\cdot\cdot
\int_{-L}^{L}}_{N \text{ terms }}\prod_{i=1}^{N}J(x_i)
W\left(s-t_0,x-\sum_{i=1}^{N}x_i\right)\\
&\ \ \ \ \times\left(\rho+r\left(x-\sum_{i=1}^{N}x_i-cs\right)
-W\left(s-t_0,x-\sum_{i=1}^{N}x_i\right)\right){\rm d}x_1\cdot\cdot\cdot{\rm d}x_N\bigg).
\end{split}
\end{equation*}
For any $t\geq t_1$ and $x$ satisfying \eqref{25},
since $r(\cdot)$ is nondecreasing and $\varphi(\lambda_1,\gamma)=c+\delta$,
we then oatain
\begin{equation*}
\begin{split}
x-\sum_{i=1}^{\tilde{N}}x_i-ct\ge& l+\varphi(\lambda_1,\gamma)(t-t_0)
+\sigma(\lambda_1,\gamma)-ct\\
=&l-\varphi(\lambda_1,\gamma)t_0+\sigma(\lambda_1,\gamma)+\delta t\\
\geq&l-\varphi(\lambda_1,\gamma)t_0+\sigma(\lambda_1,\gamma)+\delta t_0\\
=&l+\sigma(\lambda_1,\gamma)-c t_0
\geq l,
\end{split}
\end{equation*}
which implies that $r\left(x-\sum_{i=1}^{\tilde{N}}x_i-cs\right)\geq r(l)$
for $\tilde{N}=1,2,\cdot\cdot\cdot,N$. According to the assumption that
$c^*(l)=c^*(\infty)-\delta$, we have
\begin{align*}
\frac{d\left(\int_{\mathbb{R}}J(y)e^{\lambda^*(\infty) y}{\rm d}y-1\right)
+r(\infty)}{\lambda^*(\infty)}-\delta=&
\inf_{\lambda>0}\frac{d\left(\int_{\mathbb{R}}
J(y)e^{\lambda y}{\rm d}y-1\right)+r(l)}{\lambda}\\
\leq&\frac{d\left(\int_{\mathbb{R}}
J(y)e^{\lambda^*(\infty) y}{\rm d}y-1\right)+r(l)}{\lambda^*(\infty)}.
\end{align*}
It then follows that $r(l)\geq r(\infty)-\delta \lambda^*(\infty)$,
and hence $$r\left(x-\sum_{i=1}^{\tilde{N}}x_i-cs\right)\geq r(\infty)-\delta \lambda^*(\infty).$$
Similar to \eqref{26}, we can obtain that for the above $\epsilon>0$,
when $t\geq s\geq t_1$ and $x$ satisfying \eqref{25},
there exists large $N_2>0$ such that $N\geq N_2$,
\begin{equation}\label{27}
\begin{split}
&\left[e^{-\rho (t-s)}P(t-s)W(s-t_0,\cdot)(\rho+r(\cdot-cs)-W(s-t_0,\cdot))\right](x)\\
&\geq e^{-\rho (t-s)}e^{-d(t-s)}\alpha(\rho+r(\infty)-\delta \lambda^*(\infty)-\alpha)
\sum_{k=0}^{N}\frac{[d(t-s)]^k}{k!}\\
&=e^{-\rho (t-s)}\alpha(\rho+r(\infty)-\delta \lambda^*(\infty)-\alpha)
\left(1-e^{-d(t-s)}\sum_{k=N+1}^{\infty}\frac{[d(t-s)]^k}{k!}\right)\\
&\geq e^{-\rho (t-s)}\alpha(\rho+r(\infty)-\delta \lambda^*(\infty)-\alpha)(1-\epsilon).
\end{split}
\end{equation}
Let $N\geq\max\{N_1,N_2\}$. In view of \eqref{24}, \eqref{26} and \eqref{27},
we then conclude that for $t\geq t_1$ and $x$ satisfying \eqref{25},
there holds
\begin{equation*}
  u(t,x)\geq v^1(t),
\end{equation*}
where
\begin{equation*}
  v^1(t)=(1-\epsilon)\alpha e^{-\rho (t-t_1)}+(1-\epsilon)
\int_{t_1}^{t}e^{-\rho (t-s)}\alpha(\rho+r(\infty)-\delta \lambda^*(\infty)-\alpha){\rm d}s.
\end{equation*}

Using \eqref{24}, by induction, we can further derive that for sufficiently large $t\geq t_1$ and $x$ satisfying
\begin{equation}\label{28}
 l+\varphi(\lambda_1,\gamma)(t-t_0)+\sigma(\lambda_1,\gamma) +mNL \leq x
                 \leq l+\varphi(\lambda_2,\gamma)(t-t_0)
                   +\sigma(\lambda_2,\gamma)+\frac{3\pi}{\gamma}-mNL,
\end{equation}
there holds
\begin{equation*}
  u(t,x)\geq v^m(t)
\end{equation*}
with $v^m$ being defined recursively by
\begin{equation}\label{29}
\begin{split}
 v^m(t)=&(1-\epsilon)\alpha e^{-\rho (t-t_1)}\\
&+(1-\epsilon)\int_{t_1}^{t}e^{-\rho (t-s)} v^{m-1}(s)
(\rho+r(\infty)-\delta \lambda^*(\infty)- v^{m-1}(s)){\rm d}s, ~m\geq2.
\end{split}
\end{equation}

Clearly, $0\leq v^m(t)\leq r(\infty)$ for all $m\geq1$.
Next, we study the asymptotic behavior of the sequence
$\{v^m(t)\}$ as $t\rightarrow\infty$.
First, we rewrite \eqref{29} in its differential form:
\begin{equation}\label{30}
\begin{cases}
\frac{{\rm d}v^m (t)}{{\rm d} t}
=-\rho  v^m (t)+(1-\epsilon)v^{m-1}(t)
(\rho+r(\infty)-\delta \lambda^*(\infty)- v^{m-1}(t)), t>t_1,\\
v^m (t_1)=(1-\epsilon)\alpha, ~m\geq2.
\end{cases}
\end{equation}
By understanding \eqref{30} as a linear first order ODE on $v^m$
with a nonhomogeneous term involving $v^{m-1}$,
we see from the classical theory of particular solutions that
\begin{equation}\label{52}
  v^m(t)=v^m(\infty)+B_m(t)e^{-\rho(t-t_1)},
\end{equation}
where $v^m(\infty)=\lim_{t\rightarrow\infty}v^m(t)$ and $B_m(t)$
is a sum of polynomials of $t$, and products of polynomials of $t$
and exponential functions with the form of $e^{-k\rho(t-t_1)}$ for some
$k>0$. It then follows from \eqref{52} that
$\lim_{t\rightarrow\infty}\frac{{\rm d}v^m (t)}{{\rm d} t}=0$,
which, together with equation \eqref{30}, indicates that
\begin{equation*}
  -\rho  v^m (\infty)+(1-\epsilon)v^{m-1}
(\infty)(\rho+r(\infty)-\delta \lambda^*(\infty)- v^{m-1}(\infty))=0.
\end{equation*}
Let $m\rightarrow\infty$ in the above equation, we have
\begin{equation*}
  \lim_{m\rightarrow\infty}v^m (\infty)
=r(\infty)-\delta \lambda^*(\infty)-\frac{\epsilon\rho}{1-\epsilon}.
\end{equation*}
For an arbitrarily small $\iota>0$, we choose $M$ such that
\begin{equation}\label{33}
v^M (\infty)\geq r(\infty)-\delta \lambda^*(\infty)-\frac{\epsilon\rho}{1-\epsilon}-\iota.
\end{equation}
We now choose $t_1$ large enough such that for $t\geq t_1$,
\begin{equation}\label{34}
\begin{split}
 &l+\varphi(\lambda_1,\gamma)(t-t_0)+\sigma(\lambda_1,\gamma) +MNL \\
&\leq x\leq l+\varphi(\lambda_2,\gamma)(t-t_0)
+\sigma(\lambda_2,\gamma)+\frac{3\pi}{\gamma}-MNL.
\end{split}
\end{equation}
Furthermore,
\begin{equation}\label{35}
 \lim_{t\rightarrow\infty}\inf_{t\geq t_1, x \text{ satisfies }
 \eqref{34}}u(t,x)\geq v^M(\infty).
\end{equation}

For any given $\varsigma\in\left(0,\frac{c^*(\infty)-c}{2}\right)$,
pick $\delta$ small enough with $\delta<\varsigma/4$,
by Lemma \ref{lem3}, $\varphi(\lambda_1,\gamma)=c+\delta$
and $c^*(\infty)=c^*(l)+\delta\leq c^*_\gamma(l)+2\delta=\varphi(\lambda_2,\gamma)+4\delta$.
Thus, we can choose sufficiently large $t\geq t_1$ such that
\begin{equation*}
\begin{split}
&l+\varphi(\lambda_1,\gamma)(t-t_0)+\sigma(\lambda_1,\gamma) +MNL\\
&\leq(c+\varsigma)t<(c^*(\infty)-\varsigma)t\\
&\leq l+\varphi(\lambda_2,\gamma)(t-t_0)+\sigma(\lambda_2,\gamma)+\frac{3\pi}{\gamma}-MNL.
\end{split}
\end{equation*}
It follows that
\begin{equation*}
 \lim_{t\rightarrow\infty}\inf_{ (c+\varsigma)t\leq x
\leq (c^*(\infty)-\varsigma)t}u(t,x)\geq v^M(\infty)\geq r(\infty)
-\delta \lambda^*(\infty)-\frac{\epsilon\rho}{1-\epsilon}-\iota.
\end{equation*}
Since the parameters $\delta$, $\epsilon$, $\iota$ can be chosen arbitrarily small, we then obtain
\begin{equation*}
 \lim_{t\rightarrow\infty}\inf_{ (c+\varsigma)t\leq x
\leq (c^*(\infty)-\varsigma)t}u(t,x)\geq r(\infty).
\end{equation*}
On the other hand, Theorem \ref{thm1} implies that $0\leq u(t,x)\leq r(\infty)$
for all $t>0$ and $x\in\mathbb{R}$. In particular, we have
$$\lim_{t\rightarrow\infty}\sup_{ (c+\varsigma)t\leq x
\leq (c^*(\infty)-\varsigma)t}u(t,x)\leq r(\infty),$$
which implies that the statement (iii) is valid.
\end{proof}

We remark that the persistence obtained in Theorem \ref{persist-ss}
should be understood from the view of ``by moving", that is, the species
will move toward the better resource with speed $c^*(\infty)$
which is larger than the shifting speed $c$. In fact, for any given location $x$,
since the resource function $r(x-ct)$ will become negative eventually as time goes,
then the population at this location will vanish.

\section{Forced traveling waves} \label{tws}

In this section, we consider the positive traveling wave solutions of \eqref{1}
with the wave speed at which the habitat is shifting.

Letting $u(t,x)=U(\xi)$ with $\xi=x-ct$, we see from \eqref{1} that $U(\xi)$ satisfies
\begin{equation}\label{36}
 -cU'(\xi)=d\left(\int_{\mathbb{R}}J(y)U(\xi-y){\rm d}y-U(\xi)\right)+U(\xi)(r(\xi)-U(\xi)),
~\xi\in\mathbb{R},
\end{equation}
where the symbol prime stands for the derivative.
Recall that $c>0$ is the habitat shifting speed.
As explained in \cite{HuZou2017},
we impose the following boundary conditions
\begin{equation}\label{37}
 \lim_{\xi\rightarrow-\infty}U(\xi)=0, ~\lim_{\xi\rightarrow\infty}U(\xi)=r(\infty).
\end{equation}
This type of traveling waves can help us understand the point-wise ``die-out dynamics" of
the species under consideration .
By Remark \ref{rem1}, we have
$0 \leq U(\xi)\leq r(\infty)$, $\forall\xi\in\mathbb{R}$. Moreover, by the strong maximum principles for nonlocal equations \cite[Theorem 2.12]{Coville2007},
we get $0<U<r(\infty)$ in $\mathbb{R}$.

Let $V(\xi)=U(-\xi)$, $\forall \xi\in\mathbb{R}$.
It then follows from \eqref{36} and the symmetry of $J$ that
\begin{equation}\label{38}
cV'(\xi)=d\left(\int_{\mathbb{R}}J(y)V(\xi-y){\rm d}y-V(\xi)\right)+V(\xi)(r(-\xi)-V(\xi)),
~\xi\in\mathbb{R}.
\end{equation}
Correspondingly, we have
\begin{equation}\label{39}
\lim_{\xi\rightarrow-\infty}V(\xi)=r(\infty), ~\lim_{\xi\rightarrow\infty}V(\xi)=0.
\end{equation}
In the following, by the combination of super/sub-solutions and monotone iterations,
we will prove that \eqref{38} admits a nonincreasing solution
satisfying \eqref{39} for any given $c>0$, which gives rise to the nondecreasing
traveling wave solution of \eqref{1} connecting $0$ to $r(\infty)$.

In order to construct a subsolution for \eqref{38}, we introduce an auxiliary nonlocal dispersal equation of ignition type.
For any small $\varepsilon\in(0,r(\infty)/5)$, define
\begin{equation*}
f_\varepsilon(u)= \left\{
                               \begin{array}{ll}
                               u(r(\infty)-\varepsilon-u),& \text{if } u\geq 0, \\
                               0, & \text{if} -\varepsilon\leq u<0.
                               \end{array}
                             \right.
\end{equation*}
We consider the following problem:
\begin{equation}\label{48}
\frac{\partial u(t,x)}{\partial t}=d(J\ast u-u)(t,x)+f_\varepsilon(u(t,x)).
\end{equation}
According to \cite{Coville2003,Coville2012}, equation \eqref{48} admits a decreasing traveling wave
solution $V_\varepsilon(\xi)$ ($\xi=x-c_\varepsilon t$)
connecting $r(\infty)-\varepsilon$ to $-\varepsilon$ with speed $c_\varepsilon$,
that is, $(V_\varepsilon,c_\varepsilon)$ satisfies
\begin{equation}\label{49}
\begin{cases}
 -c_\varepsilon V'_\varepsilon(\xi)=d\left(\int_{\mathbb{R}}J(y)V_\varepsilon(\xi-y){\rm d}y
 -V_\varepsilon(\xi)\right)+f_\varepsilon(V_\varepsilon(\xi)),\\
V_\varepsilon(-\infty)=r(\infty)-\varepsilon, ~V_\varepsilon(+\infty)=-\varepsilon,
~V'_\varepsilon(\xi)<0.
\end{cases}
\end{equation}
Then we have the following observation.
\begin{lemma}\label{speed}
Let $c^*(\infty)$ be the minimal wave speed of the monotone traveling wave solution
connecting $r(\infty)$ to $0$
for the nonlocal dispersal Fisher-KPP equation:
\begin{equation}\label{51}
u_t=d(J\ast u-u)+u(r(\infty)-u).
\end{equation}
Then
$
  \lim_{\varepsilon\rightarrow 0^+}c_\varepsilon=c^*(\infty).
$
\begin{proof}
Integrating the first equation of \eqref{49} from $-\infty$ to $\infty$,
by the symmetry of $J$, for any $\varepsilon\in(0,r(\infty)/5)$, we have
\begin{equation}\label{50}
  c_\varepsilon=\frac{1}{r(\infty)}\int_{-\infty}^{\infty}
  f_\varepsilon(V_\varepsilon(\xi)){\rm d}\xi
  \geq\frac{1}{r(\infty)}\int_{-\infty}^{\infty}
  f_{r(\infty)/5}(V_{r(\infty)/5}(\xi)){\rm d}\xi>0.
\end{equation}
We first claim that $c_\varepsilon$ is nonincreasing in $\varepsilon>0$.
In fact, let $\varepsilon_1>\varepsilon_2>0$,
and $u_1(t,x)=V_{\varepsilon_1}(x-c_{\varepsilon_1}t)$
and $u_2(t,x)=V_{\varepsilon_2}(x-c_{\varepsilon_2}t)$ be
the decreasing traveling wave solution
of \eqref{48} with $\varepsilon$ equals $\varepsilon_1$ and $\varepsilon_2$, respectively.
Noting that $V_{\varepsilon_1}(-\infty)< V_{\varepsilon_2}(-\infty)$,
$V_{\varepsilon_1}(+\infty)< V_{\varepsilon_2}(+\infty)$
and $f_{\varepsilon_1}(u)\leq f_{\varepsilon_2}(u)$.
Since any translation of a wave profile is also a wave profile,
we can always assume that $V_{\varepsilon_1}(x)\leq V_{\varepsilon_2}(x)$,
$\forall x\in\mathbb{R}$. Then the comparison principle implies that
$V_{\varepsilon_1}(x-c_{\varepsilon_1}t)\leq
 V_{\varepsilon_2}(x-c_{\varepsilon_2}t)$, $\forall x\in\mathbb{R}$, $t>0$, and hence,
we have $c_{\varepsilon_1}\leq c_{\varepsilon_2}$.
Similarly, we can further show that
for any small $\varepsilon>0$, $c_\varepsilon\leq c^*(\infty)$.
Thus, $\lim_{\varepsilon\rightarrow 0^+}c_\varepsilon$ exists.
Set $\tilde{c}:=\lim_{\varepsilon\rightarrow 0^+}c_\varepsilon$,
we then have $\tilde{c}\leq c^*(\infty)$.

Recalling that for any $\varepsilon\in(0,r(\infty)/5)$,
$-r(\infty)/5\leq -\varepsilon \leq V_\varepsilon\leq r(\infty)-\varepsilon\leq r(\infty)$.
By \eqref{50} and the first equation of \eqref{49}, a direct computation yields that
there exists a constant $M_1>0$ such that $|V'_\varepsilon|\leq M_1$.
Moreover, differentiating \eqref{49} with respect to $\xi$,
we can get that $|V''_\varepsilon|\leq M_2$ for some $M_2>0$.
Since $c_\varepsilon\rightarrow \tilde{c}$ as $\varepsilon\rightarrow 0^+$,
by the uniform boundedness of $|V'_\varepsilon|$ and $|V''_\varepsilon|$,
there exists a sequence $\varepsilon_n\rightarrow 0$ such that
$V_{\varepsilon_n}\rightarrow \tilde{V}$ in $C^1_{loc}(\mathbb{R})$,
where $\tilde{V}$ satisfies
\begin{equation*}
 -\tilde{c} \tilde{V}'(\xi)=d\left(\int_{\mathbb{R}}J(y)\tilde{V}(\xi-y){\rm d}y
 -\tilde{V}(\xi)\right)+\tilde{V}(\xi)(r(\infty)-\tilde{V}(\xi)).
\end{equation*}
Further, $\tilde{V}$ is nonincreasing on $\mathbb{R}$ and $0\leq\tilde{V}\leq r(\infty)$.
Without loss of generality, we can normalize $V_{\varepsilon_n}$
by $V_{\varepsilon_n}(0)=\frac{r(\infty)}{2}$, it then follows that
$\tilde{V}(0)=\frac{r(\infty)}{2}$. Thus, we have $\tilde{V}(-\infty)=r(\infty)$
and $\tilde{V}(+\infty)=0$. This implies that $\tilde{V}(x-\tilde{c}t)$ is a traveling wave of \eqref{51}
connecting $r(\infty)$ to $0$ with speed $\tilde{c}$,
and hence, $\tilde{c}\geq c^*(\infty)$.
Consequently, $\tilde{c}=c^*(\infty)$.
\end{proof}
\end{lemma}

\begin{lemma}\label{lower}
Fix a sufficiently small $\varepsilon\in(0,r(\infty)/5)$. Then for any $c>-c^*(\infty)$,
$\underline{V}(\xi):=\max\{V_\varepsilon(\xi),0\}$ is a subsolution of \eqref{38},
i.e., $\underline{V}$ satisfies the following inequality:
\begin{equation}\label{47}
d\left(\int_{\mathbb{R}}J(y)\underline{V}(\xi-y){\rm d}y-\underline{V}(\xi)\right)
-c\underline{V}'(\xi)+\underline{V}(\xi)(r(-\xi)-\underline{V}(\xi))\geq0
\end{equation}
for any $\xi\neq \xi_0$, where $\xi_0$ is the point satisfying $V_\varepsilon(\xi_0)=0$ and
$V_\varepsilon$ fulfills \eqref{49}.
\end{lemma}
\begin{proof}
According to Lemma \ref{speed}, $\lim_{\varepsilon\rightarrow0^+}c_\varepsilon=c^*(\infty)$.
It then follows that for sufficiently small $\varepsilon\in (0,r(\infty)/5)$, we have
$c_\varepsilon>-c$ due to $c>-c^*(\infty)$.
Let us fix such an $\varepsilon$. Without loss generality,
we can assume that $V_\varepsilon(\xi_0)=0$ and
$r(-\xi_0)\geq r(\infty)-\varepsilon$.
This can be realized by some appropriate translation of $V_\varepsilon$ if necessary.

If $\xi<\xi_0$, $\underline{V}(\xi)=V_\varepsilon(\xi)>0$,
since $r$ is nondecreasing and $V'_\varepsilon(\xi)<0$,
we have
\begin{align*}
  &d\left(\int_{\mathbb{R}}J(y)\underline{V}(\xi-y){\rm d}y-\underline{V}(\xi)\right)
-c\underline{V}'(\xi)+\underline{V}(\xi)(r(-\xi)-\underline{V}(\xi))\\
&\geq d\left(\int_{\mathbb{R}}J(y)V_\varepsilon(\xi-y){\rm d}y-V_\varepsilon(\xi)\right)
-cV_\varepsilon'(\xi)+V_\varepsilon(\xi)(r(-\xi_0)-V_\varepsilon(\xi))\\
&\geq d\left(\int_{\mathbb{R}}J(y)V_\varepsilon(\xi-y){\rm d}y-V_\varepsilon(\xi)\right)
+c_\varepsilon V_\varepsilon'(\xi)+V_\varepsilon(\xi)(r(\infty)-\varepsilon-V_\varepsilon(\xi))\\
&=d\left(\int_{\mathbb{R}}J(y)V_\varepsilon(\xi-y){\rm d}y-V_\varepsilon(\xi)\right)
+c_\varepsilon V_\varepsilon'(\xi)+f_\varepsilon(V_\varepsilon(\xi))\\
&=0, ~~(\text{ by } \eqref{49})
\end{align*}
which shows that \eqref{47} holds for $\xi<\xi_0$.

If $\xi>\xi_0$, $\underline{V}(\xi)=0$, then
\begin{equation*}
d\left(\int_{\mathbb{R}}J(y)\underline{V}(\xi-y){\rm d}y-\underline{V}(\xi)\right)
-c\underline{V}'(\xi)+\underline{V}(\xi)(r(-\xi)-\underline{V}(\xi))
=d\int_{\mathbb{R}}J(y)\underline{V}(\xi-y){\rm d}y\geq0.
\end{equation*}
Hence, \eqref{47} also holds for $\xi>\xi_0$.
\end{proof}

Next, we construct a supersolution for equation \eqref{38}.
\begin{lemma}\label{upper}
Choose $\xi_1>\xi_0$ large enough such that $r(-\xi_1)<0$.
Let $\mu_1>0$ be the solution
of $$d\left(\int_\mathbb{R}J(y)e^{\mu y}{\rm d}y-1\right)+c\mu+r(-\xi_1)=0.$$
Then for any $c>0$, $\overline{V}(\xi):=\min\{r(\infty), r(\infty)e^{-\mu_1(\xi-\xi_1)}\}$
satisfies
\begin{equation}\label{40}
d\left(\int_{\mathbb{R}}J(y)\overline{V}(\xi-y){\rm d}y-\overline{V}(\xi)\right)
-c\overline{V}'(\xi)+\overline{V}(\xi)(r(-\xi)-\overline{V}(\xi))\leq0
\end{equation}
for any $\xi\neq\xi_1$.
\end{lemma}
\begin{proof}
Let $$h(\mu)=d\left(\int_\mathbb{R}J(y)e^{\mu y}{\rm d}y-1\right)+c\mu+r(-\xi_1).$$
By a direct calculation, we have
\begin{gather*}
  h(0)=r(-\xi_1)<0, ~h(\mu)\rightarrow +\infty \text{ as } \mu\rightarrow\infty,  \\
  h'(0)=c>0, ~h''(\mu)>0, \forall \mu\in\mathbb{R},
\end{gather*}
which implies that there exists $\mu_1>0$ such that $h(\mu_1)=0$.

According to the definition of $\overline{V}(\xi)$ and the assumption of $J$, we have
\begin{equation*}
\int_{\mathbb{R}}J(y)V(\xi-y){\rm d}y\leq \min\left\{r(\infty),
r(\infty)e^{-\mu_1(\xi-\xi_1)}\int_{\mathbb{R}}J(y)e^{\mu_1 y}{\rm d}y\right\}.
\end{equation*}
When $\xi<\xi_1$, $\overline{V}(\xi)=r(\infty)$, and hence, we have
\begin{equation*}
d\left(\int_{\mathbb{R}}J(y)V(\xi-y){\rm d}y-V(\xi)\right)-cV'(\xi)+V(\xi)(r(-\xi)-V(\xi))
\leq r(\infty)(r(-\xi)-r(\infty))\leq0.
\end{equation*}
When $\xi>\xi_1$, $\overline{V}(\xi)=r(\infty)e^{-\mu_1(\xi-\xi_1)}$. It then follows that
\begin{align*}
&d\left(\int_{\mathbb{R}}J(y)V(\xi-y){\rm d}y-V(\xi)\right)-cV'(\xi)+V(\xi)(r(-\xi)-V(\xi))\\
&\leq r(\infty)e^{-\mu_1(\xi-\xi_1)}
\left[d\left(\int_{\mathbb{R}}J(y)e^{\mu_1 y}{\rm d}y-1\right)+c\mu_1+r(-\xi)
-r(\infty)e^{-\mu_1(\xi-\xi_1)}\right]\\
&\leq r(\infty)e^{-\mu_1(\xi-\xi_1)}
\left[d\left(\int_{\mathbb{R}}J(y)e^{\mu_1 y}{\rm d}y-1\right)+c\mu_1+r(-\xi_1)\right]\\
&=0.
\end{align*}
This completes the proof of \eqref{40}.
\end{proof}

Let $BC(\mathbb{R},\mathbb{R})$ be the space of all
bounded and continuous functions from $\mathbb{R}$ to $\mathbb{R}$,
and $BC^+=\{v\in BC(\mathbb{R},\mathbb{R}): v(x)\geq0 \text{ for all } x\in\mathbb{R}\}$.
For any $v$, $\tilde{v}\in BC(\mathbb{R},\mathbb{R})$, we denote $v\geq \tilde{v}$ or
$\tilde{v}\leq v$ if $v-\tilde{v}\in BC^+$. According to Lemmas \ref{lower} and \ref{upper},
we can easily verify that $\underline{V}\leq \overline{V}$ on $\mathbb{R}$.
Define the profile set
$$\Theta=\{v\in BC(\mathbb{R},\mathbb{R}): \underline{V}\leq v\leq \overline{V}\}$$
and the operator $H: \Theta\rightarrow C(\mathbb{R},\mathbb{R})$ by
\begin{equation*}
  H(V)(\xi)=\beta V(\xi)+d\int_{\mathbb{R}}J(y)V(\xi-y){\rm d}y-dV(\xi)+V(\xi)(r(-\xi)-V(\xi)),
\end{equation*}
where $\beta=d+2r(\infty)-r(-\infty)>0$. Then \eqref{38} can be rewritten as
\begin{equation}\label{41}
 cV'(\xi)=-\beta V(\xi)+H(V)(\xi).
\end{equation}
Clearly, if $v\in\Theta$, then $0\leq v(\xi)\leq r(\infty)$ for all $\xi\in\mathbb{R}$.
Let us introduce the following integral equation
\begin{equation}\label{42}
  V(\xi)=\frac{1}{c}\int_{-\infty}^{\xi}e^{-\frac{\beta}{c}(\xi-z)}H(V)(z){\rm d}z,
\end{equation}
which is well defined for $V\in\Theta$. Moreover, it is easy to see that
the solution of \eqref{42} is $C^1$ and satisfies \eqref{41}.
Thus, the existence of monotone solutions of \eqref{38}-\eqref{39} reduces to
that of the fixed point of the operator $F:\Theta\rightarrow C(\mathbb{R},\mathbb{R})$
defined as follows
\begin{equation}\label{46}
F(V)(\xi)=\frac{1}{c}\int_{-\infty}^{\xi}e^{-\frac{\beta}{c}(\xi-z)}H(V)(z){\rm d}z.
\end{equation}

Now we summarize some properties of the operator $F$.
\begin{lemma}\label{pro}
$F$ is a nondecreasing operator and maps $\Theta$ to $\Theta$.
Moreover, if $V\in\Theta$ is nonincreasing,
then $F(V)(\xi)$ is nonincreasing with respect to $\xi$.
\end{lemma}
\begin{proof}
For any $V$, $\tilde{V}\in\Theta$ with $V\geq\tilde{V}$, we have
\begin{align*}
  H(V)(\xi)-H(\tilde{V})(\xi)
  =&[\beta-d+r(-\xi)-(V(\xi)+\tilde{V}(\xi))](V(\xi)-\tilde{V}(\xi))\\
  &+d\int_{\mathbb{R}}J(y)[V(\xi-y)-\tilde{V}(\xi-y)]{\rm d}y\\
  \geq&[r(-\xi)-r(-\infty)+2r(\infty)-(V(\xi)+\tilde{V}(\xi))](V(\xi)-\tilde{V}(\xi))\\
  \geq&0,
\end{align*}
which implies that $F(V)(\xi)\geq F(\tilde{V})(\xi)$, $\forall\xi\in\mathbb{R}$.
It then follows that
\begin{equation}\label{420}
F(\underline{V})(\xi)\leq F(V)(\xi)\leq F(\overline{V})(\xi)
\end{equation}
for all $V\in \Theta$ and all $\xi\in\mathbb{R}$.
On the other hand, by Lemma \ref{upper}, we have
\begin{equation}\label{43}
\begin{split}
  F(\overline{V})(\xi)=&\frac{1}{c}\int_{-\infty}^{\xi}
  e^{-\frac{\beta}{c}(\xi-z)}H(\overline{V})(z){\rm d}z\\
  \leq& \frac{1}{c}\left\{\left(\int_{-\infty}^{\xi_1}
  +\int_{\xi_1}^{\xi}\right)e^{-\frac{\beta}{c}(\xi-z)}
  [c\overline{V}'(z)+\beta \overline{V}(z)]{\rm d}z\right\}\\
  =&\overline{V}(\xi),
\end{split}
\end{equation}
and by Lemma \ref{lower}, we have
\begin{equation}\label{44}
\begin{split}
  F(\underline{V})(\xi)=&\frac{1}{c}\int_{-\infty}^{\xi}
  e^{-\frac{\beta}{c}(\xi-z)}H(\underline{V})(z){\rm d}z\\
  \geq& \frac{1}{c}\left\{\left(\int_{-\infty}^{\xi_0}
  +\int_{\xi_0}^{\xi}\right)e^{-\frac{\beta}{c}(\xi-z)}
  [c\underline{V}'(z)+\beta \underline{V}(z)]{\rm d}z\right\}\\
  =&\underline{V}(\xi).
\end{split}
\end{equation}
Combining \eqref{420}-\eqref{44}, we obtain $F(\Theta)\subseteq \Theta$.

If $V\in\Theta$ is nonincreasing, then for all $\xi\in\mathbb{R}$ and any $h>0$, we have
\begin{align*}
 &H(V)(\xi+h)-H(V)(\xi)\\
 &=[\beta-d-V(\xi+h)-V(\xi)](V(\xi+h)-V(\xi))
 +r(-\xi-h)V(\xi+h)-r(-\xi)V(\xi)\\
 &\ \ \ \ +d\int_{\mathbb{R}}J(y)[V(\xi+h-y)-V(\xi-y)]{\rm d}y\\
 &\leq[2r(\infty)-V(\xi+h)-V(\xi)+r(-\xi)-r(-\infty)](V(\xi+h)-V(\xi))\\
 &\leq0,
\end{align*}
which further leads to
\begin{align*}
F(V)(\xi+h)=&\frac{1}{c}\int_{-\infty}^{\xi+h}e^{-\frac{\beta}{c}(\xi+h-z)}H(V)(z){\rm d}z\\
=&\frac{1}{c}\int_{0}^{\infty}e^{-\frac{\beta}{c}z}H(V)(\xi+h-z){\rm d}z\\
\leq&\frac{1}{c}\int_{0}^{\infty}e^{-\frac{\beta}{c}z}H(V)(\xi-z){\rm d}z\\
=&\frac{1}{c}\int_{-\infty}^{\xi}e^{-\frac{\beta}{c}(\xi-z)}H(V)(z){\rm d}z=F(V)(\xi).
\end{align*}
We then complete the proof.
\end{proof}

Now we are in a position to prove our main result in this section.
\begin{theorem}
For any given $c>0$, \eqref{36} admits a nondecreasing positive solution $U(\xi)$
satisfying \eqref{37}. In other words, \eqref{1} has a  nondecreasing
forced traveling wave connecting
$0$ and $r(\infty)$ with the wave speed at which the environment is shifting.
If,  in addition, $r$ is strictly increasing, then the wave profile is monotone increasing.
\end{theorem}
\begin{proof}
Define the iterations:
\begin{equation*}
  V_1=F(\overline{V}), ~V_{n+1}=F(V_n), \forall n\geq1.
\end{equation*}
Since $\overline{V}\in\Theta$ is nonincreasing on $\mathbb{R}$,
by Lemma \ref{pro}, we can conclude that $V_n\in \Theta$ and $V_n(\xi)$
is nonincreasing with respect to $\xi$ for each fixed $n=1,2,\cdot\cdot\cdot$, and
\begin{equation*}
  \underline{V}(\xi)\leq V_{n+1}(\xi)\leq V_n(\xi)\leq \overline{V}(\xi), ~
  \forall\xi\in\mathbb{R}, n\geq1.
\end{equation*}
Then the pointwise limit of the sequence $\{V_n\}$ exists, denoted by $V$,
i.e., for every $\xi\in\mathbb{R}$, $V(\xi)=\lim_{n\rightarrow\infty}V_n(\xi)$.
Obviously, $V(\xi)$ is a nonincreasing and nonnegative function defined on $\mathbb{R}$ and
\begin{equation}\label{45}
 \underline{V}(\xi)\leq V(\xi)\leq \overline{V}(\xi).
\end{equation}
Moreover, $H(V_n)$ converges pointwise to $H(V)$.

We now show $V$ is a fixed point of $F$. Since
\begin{equation*}
  |H(V_n)|\leq [\beta+2d+r(\infty)+\max\{r(\infty),-r(-\infty)\}]r(\infty), \forall n\geq1,
\end{equation*}
by \eqref{46} and the Lebesgue's dominated convergence theorem, we have
\begin{align*}
  V(\xi)=&\lim_{n\rightarrow\infty}V_{n+1}(\xi)=\lim_{n\rightarrow\infty}F(V_{n})(\xi)\\
  =&\frac{1}{c}
  \lim_{n\rightarrow\infty}\int_{-\infty}^{\xi}e^{-\frac{\beta}{c}(\xi-z)}H(V_n)(z){\rm d}z\\
  =&\frac{1}{c}\int_{-\infty}^{\xi}e^{-\frac{\beta}{c}(\xi-z)}H(V)(z){\rm d}z=F(V)(\xi).
\end{align*}
It easily follows that $V\in C^1(\mathbb{R})$ satisfies \eqref{38}.
Next we prove that $V$ meets \eqref{39}.
Clearly, it follows from \eqref{45} as well as
\begin{equation*}
  \lim_{\xi\rightarrow\infty}\underline{V}(\xi)
  =\lim_{\xi\rightarrow\infty}\overline{V}(\xi)=0,
\end{equation*}
that
\begin{equation*}
\lim_{\xi\rightarrow\infty}V(\xi)=0.
\end{equation*}
Note that $V(\xi)$ is nonincreasing in $\mathbb{R}$ and
$0\leq \underline{V}(\xi)\leq V(\xi)\leq \overline{V}(\xi)\leq r(\infty)$.
Therefore, $A:=\lim_{\xi\rightarrow-\infty}V(\xi)$ exists. Then
\begin{align*}
  &\lim_{\xi\rightarrow-\infty}H(V)(\xi)\\
  &=\lim_{\xi\rightarrow-\infty}\left[\beta V(\xi)
  +d\int_{-L}^{L}J(y)V(\xi-y){\rm d}y-dV(\xi)+V(\xi)(r(-\xi)-V(\xi))\right]\\
  &=\beta A+A(r(\infty)-A),
\end{align*}
where $[-L,L]=supp(J)$.
Applying the L'H\v{o}pital's rule, we can get that
\begin{align*}
A=&\lim_{\xi\rightarrow-\infty}V(\xi)=\lim_{\xi\rightarrow-\infty}F(V)(\xi)\\
=&\frac{1}{c}\lim_{\xi\rightarrow-\infty}
\int_{-\infty}^{\xi}e^{-\frac{\beta}{c}(\xi-z)}H(V)(z){\rm d}z\\
=&\frac{1}{c}\lim_{\xi\rightarrow-\infty}\frac{H(V)(\xi)}{\beta/c}\\
=&A+\frac{A(r(\infty)-A)}{\beta},
\end{align*}
which derives that $A=0$ or $A=r(\infty)$.
However, if $A=0$, by the monotonicity of $V$, there must be $V\equiv0$,
which contradicts to $V\geq\underline{V}$ and the definition of $\underline{V}$.
Thus, we have $$\lim_{\xi\rightarrow-\infty}V(\xi)=r(\infty).$$
Now we can obtain the desired conclusion by using the relation $U(\xi)=V(-\xi)$.

In the case where $r$ is strictly increasing, we assume, by contradiction, that there exist
$\xi_2<\xi_3$ such that $U(\xi_2)=U(\xi_3)$. Since $U(\xi)$ is nondecreasing, we have
\begin{equation*}
  U'(\xi)=0 \text{ and }U(\xi)\equiv U(\xi_2), \forall\xi\in[\xi^+_2,\xi^-_3],
\end{equation*}
where $\xi_2<\xi^+_2<\xi^-_3<\xi_3$.  In view of  \eqref{36}, it then follows that  
\begin{equation}\label{53}
  d\int^{L}_{-L}J(y)[U(\xi^+_2-y)-U(\xi^-_3-y)]{\rm d}y
  +U(\xi^+_2)[r(\xi^+_2)-r(\xi^-_3)]=0,
\end{equation}
where $[-L,L]=supp(J)$.
However, since $0<U<r(\infty)$ is nondecreasing and $r$ is strictly increasing, we obtain
\begin{gather*}
  \int^{L}_{-L}J(y)[U(\xi^+_2-y)-U(\xi^-_3-y)]{\rm d}y\leq 0,\\
  U(\xi^+_2)[r(\xi^+_2)-r(\xi^-_3)]<0,
\end{gather*}
which contradicts  \eqref{53}. This shows that $U$ is monotone increasing.
\end{proof}

\

\noindent
{\bf Acknowledgments.}  We are grateful to Dr. Jian Fang
for his helpful discussion on the construction of the subsolution for equation \eqref{38}.
W.-T. Li was partially supported by NSF of China (11671180, 11731005)
and FRFCU (lzujbky-2017-ct01).
J.-B. Wang would like to thank the China Scholarship Council (201606180060)
for financial support during the period of his overseas study and
to express his gratitude to the Department of Mathematics and Statistics,
Memorial University of Newfoundland for its kind hospitality.
X.-Q. Zhao was partially supported by the NSERC of Canada.


\end{document}